\documentclass[12pt]{article}
\usepackage[top=1.25in, bottom=1.5in, left=1.25in, right=1.25in]{geometry}
\usepackage{fancyhdr}
\usepackage{amsfonts}
\usepackage{amsthm}
\usepackage{amsmath}
\usepackage{amssymb}
\usepackage{latexsym}
\usepackage{setspace}
\usepackage{multicol}
\usepackage{graphicx}
\usepackage{multirow}
\usepackage{float}
\usepackage[pdftex,pagebackref,hypertexnames=false, colorlinks, citecolor=black, linkcolor=blue, urlcolor=red]{hyperref}
\usepackage{listings}
\usepackage{color}
\usepackage{verbatim}
 \usepackage [latin1]{inputenc}
\newcommand{\RR}{\mathbb{R}}
\newcommand{\CC}{\mathbb{C}}
\newcommand{\QQ}{\mathbb{Q}}
\newcommand{\ZZ}{\mathbb{Z}}

\newcommand{\bc}{\mathbf{c}}

\newcommand{\wt}{\widetilde}

\newcommand{\vv}{\mathcal{V}}
\newcommand{\mcT}{\mathcal{T}}
\newcommand{\mcH}{\mathcal{H}}

\newcommand{\mcQ}{\mathcal{Q}}
\newcommand{\mcV}{\mathcal{V}}

\DeclareMathOperator{\conv}{conv}
\DeclareMathOperator{\MV}{MV}
\DeclareMathOperator{\vol}{vol}
\DeclareMathOperator{\supp}{supp}

\DeclareMathOperator{\newt}{Newt}

\newtheorem{thm}{Theorem}[section]
\newtheorem{lem}[thm]{Lemma}
\newtheorem{conj}[thm]{Conjecture}
\newtheorem{prop}[thm]{Proposition}
\newtheorem{cor}[thm]{Corollary}
\theoremstyle{definition}

\newtheorem{example}[thm]{Example}
\newtheorem{rmk}[thm]{Remark}

\newenvironment{exm}
  {\pushQED{\qed}\example}
  {\popQED\endexample}

\usepackage{tikz}
\usetikzlibrary{arrows, positioning,chains,fit,shapes,calc}
\usetikzlibrary{decorations.pathmorphing}
\tikzset{
  text style/.style={
    sloped, 
    text=black
  }}
\usetikzlibrary{arrows,chains,matrix,positioning,scopes}
\makeatletter
\tikzset{join/.code=\tikzset{after node path={%
\ifx\tikzchainprevious\pgfutil@empty\else(\tikzchainprevious)%
edge[every join]#1(\tikzchaincurrent)\fi}}}
\makeatother
\tikzset{>=stealth',every on chain/.append style={join},
         every join/.style={->}}
\tikzstyle{labeled}=[execute at begin node=$\scriptstyle,
   execute at end node=$]

\allowdisplaybreaks

\begin{document}
\title{The steady-state degree and mixed volume of a chemical reaction network}
\newcommand*\samethanks[1][\value{footnote}]{\footnotemark[#1]}
\author{
Elizabeth Gross \\ University of Hawai'i at Manoa \and
Cvetelina Hill \\ Georgia Institute of Technology }
\maketitle

\begin{abstract}
  \noindent
The steady-state degree of a chemical reaction network is the number
of complex steady-states, which is a measure of the algebraic complexity of solving the steady-state system. In general, the
steady-state degree may be difficult to compute. Here, we give an
upper bound to the steady-state degree of a reaction network by
utilizing the underlying polyhedral geometry associated with the
corresponding polynomial system. We focus on three case studies of
infinite families of networks, each generated by joining
smaller networks to create larger ones. For each family, we give a
formula for the steady-state degree and
the mixed volume of the corresponding polynomial system. 
\end{abstract}

  
  

\section{Introduction}
\label{sec:intro}

Chemical reaction networks (CRNs), under the assumption of mass-action
kinetics, are deterministic polynomial systems commonly used in
systems biology to model mechanisms such as inter- and intracellular
signaling. In this paper, we study the Newton polytopes of the steady-state system of several reaction networks. The geometry of these polytopes can inform us about the steady-state degree of the network, and consequently, the algebraic complexity of exploring regions of multistationarity. 

One way to evaluate whether a given reaction network is an
appropriate model for a biological process is to consider its capacity
for multiple positive real steady-states.  If a reaction network has
this capacity, we call the network \emph{multistationary}. Multistationarity for reaction networks with mass-action kinetics has been extensively studied (see \cite{JS2015}) with algebraic methods playing a key role \cite{Dic2016}.  

Once multistationarity is established, then bounds on the number of real positive steady-states \cite{BDG2018} \cite{FH2018} \cite{MFRCSD2016} \cite{OSTT2019}  and the regions of multistationarity can be explored \cite{CFMW2017} \cite{CIK2018} \cite{GBD2018} \cite{GHRS16}.  One method to explore regions of multistationarity, which is used in \cite{GHRS16} and \cite{CIK2018}, is to sample parameters in a systematic way and repeatedly solve the \emph{steady-state system}.  The steady-state system of a reaction network is the parameterized polynomial system formed by the steady-state equations and the conservation equations.  Solving steady-state systems can be done numerically using solvers based on polynomial homotopy continuation such as \emph{Bertini} \cite{BHSW2013}, \emph{PHCpack} \cite{V99}, and \emph{HOM-4-PS2} \cite{LTC2008}. Such solvers will return all complex solutions, and so a final step requires filtering for real, positive solutions.  We call the number of complex steady-states for generic rate constants and initial conditions the \emph{steady-state degree} of a chemical reaction network.  The steady-state degree is not only a bound on the number of real, positive steady-states, but is also a measure of the algebraic complexity of solving the steady-state system for a given reaction network. The steady-state degree is similar to the maximum likelihood degree studied in algebraic statistics \cite{CHKS2006} and the Euclidean distance degree studied in optimization \cite{DHOST2016}; the former is a measure of the algebraic complexity of maximum likelihood estimation and the latter is a measure of the algebraic complexity of minimizing the distance between a point and a variety. From the viewpoint of using numerical algebraic geometry to explore regions of multistationarity, the steady-state degree is the number of paths that need to be tracked when using a parameter-homotopy to solve the steady-state system 
and can serve as a stopping criterion for monodromy-based solvers, such as the one described in \cite{DHJLLS2018}. 

Using the steady-state degree as motivation, in this paper we study the polyhedral geometry associated
to the steady-state and conservation equations. In many cases, particularly when there are many variables involved,
the steady-state degree of a family of networks can be difficult to establish. However, we
can provide an upper bound by the B\'ezout
bound and, in the absence of boundary solutions, the mixed volume of the polynomial system arising from the 
chemical reaction network. 
As an example, the mixed volume was used to bound the steady-state degree of a model of ERK regulation in \cite{OSTT2019}. In this paper, we explore the mixed volumes of reaction networks further, giving formulas for three families of networks.  In particular, we study the combinatorics of the Newton polytopes and their Minkowski sums that arise for three infinite families of networks.

The three infinite families of chemical reaction networks that we study are constructed by successively building on smaller networks to create larger ones. The base network for each family is: the cell death model from \cite{H-H},
the Edelstein network \cite{M-P-V}, and the one-site phosphorylation
cycle (see for example, motif (a) in \cite{F-W}). For each network, we compute the mixed volume and steady-state
degree of the networks using various techniques such as explicit
computation, reducing to semi-mixed and unmixed volume computation
\cite{C}, and in the case of a randomized system, constructing a
unimodular triangulation.

\begin{table}[h]
\centering
\begin{tabular}{| c | c | c | c |}
\hline
{\bf CRN family}  & {\bf B\'ezout bound} & {\bf Mixed volume} & {\bf Steady-state degree}\\
\hline 
Cluster model & $n$ & $n-2$ & $n$ (includes two\\
for cell death & & & boundary sols)\\
\hline 
Edelstein & $2^{n+1}$ & $3$ & $3$ \\
& & & \\
\hline 
Multisite distributive & $2^{3n+1}$ & $\frac{(n+1)(n+4)}{2} - 1$ & Conjecture: $2n+1$ \\
phosphorylation & & & \\
\hline 
\end{tabular}
\caption{Summary of theorems, propositions, and conjectures on the families of chemical reaction networks studied in this paper. 
See Theorems~\ref{thm:EdelMV},~\ref{thm:EdelSSD}, and~\ref{thm:one-siteMV}; 
Propositions ~\ref{prop:cellDeathBezout}, ~\ref{prop:cellDeathMV}, ~\ref{prop:cellDeathSSD}, ~\ref{prop:EdelBezout}, and \ref{prop:BezoutPC}; 
and Conjecture~\ref{conj:ssdPC}.
}
\label{tab:summary}
\end{table}

As shown in Table \ref{tab:summary} each of these examples illustrate a different relationship between the steady-state degree and the mixed volume of the the steady-state system.  For the first family, based on a cluster model for cell death, we see that that the steady-state degree is actually slightly larger than the mixed volume, due to the presence of boundary steady-states. In the second family, based on the Edelstein model, the mixed volume and steady-state degree agree.  In the third family, multisite distributive phosphorylation, we see that the mixed volume is quadratic in the number of sites, while the steady-state degree is linear in the number of sites.

The most significant of these three case studies is the exploration of the multisite distributive phosphorylation system in Section \ref{subsec:one-site}.  The $n$-site distributive phosphorylation system can be obtained by successively gluing together $n$ copies of the one-site phosphorylation cycle \cite{G-H-M-S}.  The regions of multistationarity of this network have been been recently investigated (e.g. see \cite{BDG2018}, \cite{CIK2018}, \cite{HFC2013}) in the field of chemical reaction network theory. In addition, the number of real positive solutions has been well-studied.  For example, the authors of \cite{Wang2008} show that the number of real positive solutions is bounded above by $2n-1$ and below by $n+1$ when $n$ is even and $n$ when $n$ is odd. Furthermore, the authors of \cite{FHC2014} show that the $2n-1$ bound can be achieved when $n=3$ and $n=4$, while the authors of \cite{GRMD2019} describe parameter regions where the steady-state system has $n+1$ real positive solutions when $n$ is even and $n$ when $n$ is odd.  In Section \ref{subsec:one-site}, we give the mixed volume of the {\em randomized} steady-state system of $n$-site distributive phosphorylation.  The randomized system is a square system obtained from the overdetermined steady-state system by taking random combinations of the polynomials. Determining the mixed volume requires computing the normalized volume of a $3n+3$ dimensional $0-1$ polytope with $5n+4$ vertices and $3n +7$ facets.  At the end of Section \ref{subsec:one-site} we show that this polytope of interest is the matching polytope of a graph.

The paper is organized as follows. In Section \ref{sec:back}, we give the necessary background, definitions,
and motivation.  In Section \ref{sec:res}, we systematically explore each of the three families of networks. 


\section{Background \& motivation}
\label{sec:back}

 A {\em chemical reaction network} $\mathcal N = (\mathcal{S},
 \mathcal{C},\mathcal{R})$ is a triple where $\mathcal{S} =
 \{A_1,A_2,\ldots,A_n\}$ is a set of $n$ chemical {\em species},  $\mathcal{C} = \{y_1,y_2,\ldots,y_p\}$ is a set of $p$ {\em complexes} (finite nonnegative-integer combinations of the species), and  $\mathcal{R} = \{y_i\rightarrow y_j~|~y_i, y_j \in\mathcal{C} \}$ is a set of $r$ {\em reactions}.
  
Each complex in $\mathcal{C}$ can be written in the form $y_{i1}A_1+y_{i2}A_2+\cdots+y_{in}A_n$ where $y_{ij} \in \mathbb Z_{\geq 0}$, and thus, we will view the elements of $\mathcal{C}$ as vectors in $\mathbb Z_{\geq 0}^n$, i.e. $y_i = (y_{i1}, y_{i2}, \ldots, y_{in})$  Additionally, to each complex
of the chemical reaction network, we associate a monomial
$x^{y_i}=x_{A_1}^{y_{i1}}x_{A_2}^{y_{i2}}\cdots x_{A_n}^{y_{in}}$
where $x_{A_i}=x_{A_i}(t)$ represents the concentration for species
$A_i$ with respect to time. For example, for the reaction
$A+B\rightarrow4B+C,$ the monomials corresponding to the
reactant $A+B$ and the product $4B+C$ are $x_Ax_B$ and
$x_B^4x_C,$ respectively, with exponent vectors
$y_1 = (1,1,0)$ and $y_2 = (4,0,1).$

Let $y_i\rightarrow y_j$ be the reaction from the $i$-th
to the $j$-th complex. To each reaction we associate
a {\em reaction vector} $y_j-y_i$ that gives the net change
in each species due to the reaction. Moreover, each reaction
has an associated positive reaction rate constant $k_{ij}.$
Given a chemical reaction network
$(\mathcal{S},\mathcal{C},\mathcal{R})$ and a choice of
$k_{ij}\in\RR_{>0}^r$  the system of polynomial ordinary differential equations under 
the assumption of mass-action kinetics is
\begin{equation}
  \label{eq:ODE}
  \dfrac{dx}{dt} = \sum_{y_i\rightarrow y_j\in\mathcal{R}}
  k_{ij}x^{y_i}(y_j-y_i) = : f(x),~~~x\in\RR^n. 
\end{equation}
Setting the left-hand side of the ODEs above equal to zero gives us a
set of polynomial equations that we call the \emph{steady-state equations}.

The stoichiometric subspace associated with the chemical
reaction network $\mathcal N=( \mathcal S, \mathcal C, \mathcal R)$ is a vector subspace of $\RR^n$ spanned
by the reaction vectors $y_j-y_i,$ denoted by 
\begin{equation}
  \label{eq:stoich}
  S_{\mathcal N} := \RR\{y_j-y_i ~|~ y_i\rightarrow y_j\in\mathcal{R} \}.
\end{equation}
Given initial conditions $\mathbf{c} \in \RR^n$, the stoichiometric compatibility class is the affine space $S_{\mathcal N} + \mathbf{c} $, and the \emph{conservation equations} of $\mathcal N$ are the set of linear equations defining $S_{\mathcal N} + \mathbf{c} $. 

In this paper, we are concerned with the parameterized system of
equations formed by the steady-state and conservation equations, which we call the \emph{steady-state system}, we view the polynomials of the steady-state system as polynomials in the ring $\mathbb Q({\mathbf k}, {\mathbf c} )[x_1, \ldots x_n]$. When
the solution set of this polynomial system is zero-dimensional for generic parameters
$\mathbf{k}$ and $\mathbf{c}$, we
define the number of complex solutions to the system for generic parameters
as the \emph{steady-state degree} of $\mathcal N,$ where we
distinguish \emph{boundary} steady-states as complex solutions
$x\in\CC^n$ such that $x_i=0,$ for one or more $i=1,\ldots,n.$


The steady-state degree can be computed symbolically (using Gr\"{o}bner bases) or numerically (using polynomial homotopy continuation), however,  both these methods become computationally expensive when a large number of species are involved. In such cases we would like to know an
upper bound on the degree. Two such bounds are the B\'ezout 
bound and the Bernhtein-Kushnirenko-Khovanskii (BKK) bound.  Given a zero-dimensional polynomial system
$P=(f_1,\ldots,f_m)$ with $f_i \in \mathbb Q[x_1, \ldots, x_n]$, the B\'ezout bound on the number of solutions in $\mathbb C^n$ is the
product of the degrees of all the polynomials in the system. The BKK bound on the number of solutions in $(\mathbb C^*)^n$ is the mixed volume of $P$,
which requires $P$ to be a square system, i.e., a system of $n$ equations in
$n$ variables, in this case, $m=n$. The mixed volume of $P$ is the mixed volume of the
{\em Newton polytopes} of $f_1,\ldots,f_n,$ i.e., it is the coefficient of
the term $\lambda_1\cdots\lambda_n$ in the expansion of
$\vol(\lambda_1\newt(f_1)+\cdots+\lambda_n\newt(f_n)).$ Chen provides
sufficient conditions under which the mixed volume of the Newton
polytopes is the normalized volume of the convex hull of their
union. We state these results below and reference them later in this
note.
\begin{thm}\cite{C}
  \label{ChenThm1}
  For finite sets $S_1,\ldots,S_n\subset\QQ^n,$ let
  $\wt{S}=S_1\cup\cdots\cup S_n.$ If for every proper positive
  dimensional face $F$ of $\conv(\wt{S})$ we have $F\cap
  S_i\neq\emptyset$ for each $i=1,\ldots,n$ then $\MV(\conv
  S_1,\ldots, \conv S_n) = n!\vol_n(\conv(\wt{S})).$
\end{thm}
\begin{thm}\cite{C}
  \label{ChenThm2}
  Given $n$ nonempty finite sets $S_1,\ldots,S_n\subset\QQ^n,$ let
  $\wt{S}=S_1\cup\cdots\cup S_n.$ If every positive dimensional face
  $F$ of $\conv(\wt{S})$ satisfies one of the following conditions:
  \begin{enumerate}
  \item[(i)] $F\cap S_i\neq\emptyset$ for all $i\in\{1,\ldots,n\};$
    \item[(ii)] $F\cap S_i$ is a singleton for some
      $i\in\{1,\ldots,n\};$
      \item[(iii)] For each $i\in I:=\{i~|~F\cap S_i\neq\emptyset\},
        F\cap S_i$ is contained in a common coordinate subspace of
        dimension $|I|,$ and the projection of $F$ to this subspace is
        of dimension less than $|I|;$
      \end{enumerate}
      then $\MV(\conv S_1,\ldots,\conv S_n) = n!\vol_n(\conv(\wt{S})).$
    \end{thm}
\begin{cor} \cite{C}
  \label{ChenCor1}
  Given nonempty finite set $S_{i,j}\subset\QQ^n$ for $i=1,\ldots,m$
  and $j=1,\ldots,k_i$ with $k_i\in\ZZ^+$ and $k_1+\cdots+k_m=n,$ let
  $Q_{i,j} =\conv(S_{i,j}), \wt{S}_i = \bigcup_{j=1}^{k_i}S_{i,j},$
  and $\wt{Q}_i = \conv(\wt{S}_i).$ If for each $i,$ every positive
  dimensional face of $\wt{Q}_i$ intersecting $S_{i,j},$ for some $j,$
  on at least two points also intersects all $S_{i,1},\ldots,S_{i,k},$
  then
  \begin{equation*}
    \MV(Q_{1,1},\ldots,Q_{m,k_m}) = \MV(\underbrace{\wt{Q}_1,\ldots,\wt{Q}_1}_{k_1},\ldots,\underbrace{\wt{Q}_m,\ldots,\wt{Q}_m}_{k_m}).
  \end{equation*}
\end{cor}

In this collection of case studies, for each family of networks, we
give the steady-state degree, the B\'ezout bound, and the mixed volume
of the steady-state systems employing these results and other standard
techniques.

\section{Three families of networks}
\label{sec:res}
In what follows, we investigate three infinite families of reaction networks.  The second two families result from successively joining, or gluing, smaller networks
to form a larger network as defined in \cite{G-H-M-S}.  

The first two families in this study showcase different methods that
can be used to understand the steady-state degree, while the third
family, mulitisite distributive sequential phosphorylation, requires more
sophisticated methods.  In particular, in the third case study, we
describe the polytope $Q_n$ obtained by taking the convex hull of the
exponent vectors of the support of the system.  We compute the
normalized volume of $Q_n$, which bounds the number of non-boundary steady-states.  This computation is done by
first establishing the $\mathcal{H}$-representation of $Q_n$ and then explicitly constructing a regular unimodular triangulation of $Q_n$. 

\subsection{Cell death model}
\label{subsec:cellDeath}
The first case study is a model representing the
cell death mechanism as described in \cite{H-H}.
We consider these cluster-stabilization reactions involving unstable and
stable open receptors, where each network of the family has two species: $Y$ and
$Z$, the unstable and stable receptors, respectively. The $n$th reaction network in this family,
denoted $CD_n$, has $n$ complexes $C_i$
of the form $C_i=(n-i)Y+iZ$, with $i=1,\ldots, n,$ and $n(n-1)/2$ reactions $C_i
\xrightarrow{~~k_{i,j}~~} C_j$ such that $i<j.$ The polynomial system
associated to $CD_n$ consists of one
linear conservation equation in the variables $x_Y$ and $x_Z$ and
their initial conditions $\bc = (c_Y, c_Z)$ and two steady-state equations, one for
each species. Specifically, the polynomial system
of interest is
\begin{equation}
  \begin{split}
    \label{eq:cellDeathGeneral}
    f_1 & = x_Y+x_Z-c_Y-c_Z \\ 
f_2 & = \dot{x}_Y  = -\sum_{{i,j,i\neq j}}^{n}(j-i)k_{i,j}x_Y^{j}x_Z^{n-j} \\
f_3 & = \dot{x}_Z = \sum_{i, j, i\neq j}^{n}(j-i)k_{i,j}x_Y^{j}x_Z^{n-j}.
\end{split}
\end{equation}
Since $\dot{x}_Z = -\dot{x}_Y,$ there is
only one unique steady-state equation of degree $n.$
In this example, both the B\'ezout and BKK bounds are linear in
$n,$ with the BKK bound being slightly lower. In
Proposition \ref{prop:cellDeathSSD} we show that the steady-state
degree, including boundary solutions, is given by
the B\'ezout bound; see Remark \ref{rmk:cellDeath}.

\begin{exm}
  \label{ex:CD4}
For $n=4,$ the cell death model has two species, four complexes, and
six reactions. Figure \ref{f:cd4} shows the reaction graph for this
model.
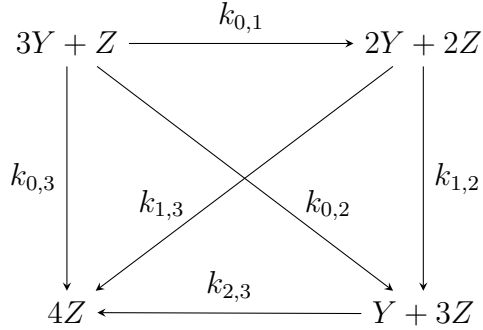
\begin{figure}
    \centering   
        \begin{tikzpicture}
          \matrix (m) [matrix of math nodes, row sep=3cm, column sep=3cm]
          {3Y+Z  & 2Y+2Z \\
            4Z & Y+3Z \\ };
          \path[-stealth]
          (m-1-1) edge node [above] {$k_{0,1}$} (m-1-2)
                        edge node [left] {$k_{0,3}$} (m-2-1)
                        edge node [right=+1.1, below] {$k_{0,2}$} (m-2-2)
          (m-1-2) edge node [right] {$k_{1,2}$} (m-2-2)
                        edge node [left=+1.1, below] {$k_{1,3}$}
                        (m-2-1)
          (m-2-2) edge node [above] {$k_{2,3}$} (m-2-1);
        \end{tikzpicture}
    \caption{A chemical reaction network of type $CD_4$ with 4
      complexes and 6 reactions.}
    \label{f:cd4} 
  \end{figure}
  The polynomial system for $CD_4$ consists of one conservation
  equation and two steady-state equations, as displayed below. 
  \begin{equation}
    \begin{split}
          \label{ps:cd4}
f_1 & = x_Y+x_Z-c_Y-c_Z \\
f_2 = \dot{x_Y} & =
-k_{0,1}x_Y^3x_Z-2k_{0,2}x_Y^3x_Z-3k_{0,3}x_Y^3x_Z\\
&~~~~~~~~~~~~~~~~~~~~ -k_{1,2}x_Y^2x_Z^2-2k_{1,3}x_Y^2x_Z^2-k_{2,3}x_Yx_Z^3\\
f_3 = \dot{x_Z} & =
k_{0,1}x_Y^3x_Z+3k_{0,2}x_Y^3x_Z+2k_{0,3}x_Y^3x_Z\\
&~~~~~~~~~~~~~~~~~~~~ +k_{1,2}x_Y^2x_Z^2+2k_{1,3}x_Y^2x_Z^2+k_{2,3}x_Yx_Z^3.
\end{split}
\end{equation}
Observe that $f_3=-f_2,$ hence we have a square system in two
variables.
  \begin{figure}
\begin{multicols}{2}
\centering
  \includegraphics[scale=0.25]{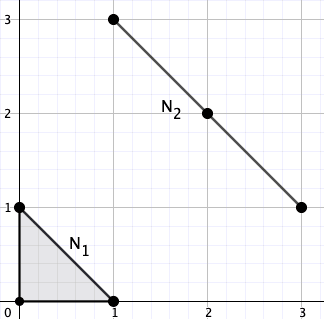}
  \caption{Newton polytopes for the polynomials corresponding to $CD_4$
  in Example \ref{ex:CD4}.}
\label{fig:KnN1N2}
  \includegraphics[scale=0.2]{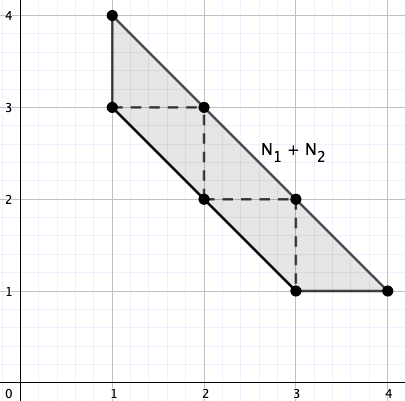}
  \caption{Minkowski sum  of the Newton polytopes for the system in
    Example \ref{ex:CD4}.}
  \label{fig:KnMink}
   \end{multicols}
  \end{figure}
\end{exm}

\begin{prop}
  \label{prop:cellDeathBezout}
The B\'ezout bound for the chemical reaction network $CD_n$ is
$n.$
\end{prop}

\begin{proof}
The B\'ezout bound can be seen from the system (\ref{eq:cellDeathGeneral}) --
there are always three equations, one linear and two of degree $n.$
However, the two degree $n$ equations are identical, hence we have
two equations and the B\'ezout bound is $n.$
\end{proof}

\begin{prop}
  \label{prop:cellDeathMV}
The polynomial system corresponding to the chemical reaction network $CD_n$ has mixed volume $n-2.$
\end{prop}
\noindent
The proof of this result requires a direct computation of the mixed
volume of the system. There are two Newton polytopes for any $n,$ one
of which is a line segment. Hence, the computation is straightforward. Recall
that the mixed volume of $m$ polytopes $Q_1,\ldots,Q_m\subset\RR^n$ is
$\MV(Q_1,\ldots,Q_m),$ which is the
coefficient of $\lambda_1\lambda_2\cdots\lambda_m$ in the expansion of
$\vol_n(\lambda_1Q_1+\lambda_2Q_2+\cdots+\lambda_mQ_m),$
with  $\lambda_i\geq0.$

\begin{proof}
Consider the system (\ref{eq:cellDeathGeneral}) for a network of type $CD_n$ for
some $n>1.$ As discussed earlier, we can consider only the first two
polynomials $f_1$ and $f_2,$ whose Newton polytopes in $\RR^2$ are
\begin{equation}
  \begin{split}
    \label{eq:cdNP}
    N_1 & = \conv((1,0),(0,1),(0,0))\\
    N_2 & = \conv((1,n-1),(2,n-2),\ldots,(n-2,2),(n-1,1))
  \end{split}
\end{equation}
Note that $N_1$ is a triangle of area $\frac{1}{2}$ and $N_2$ is
a line of length $\sqrt{2}(n-2).$ In this case, the mixed volume of
(\ref{eq:cellDeathGeneral}) is the coefficient of $\lambda_1\lambda_2$
in the following expansion
\begin{equation}
  \begin{split}
    \label{MinkowskiExp}
    \vol_2(\lambda_1N_1+\lambda_2N_2) =
    \vol_2(N_1)\lambda_1^2 & +2\vol_2(N_1,N_2)\lambda_1\lambda_2\\
    & +
    \vol_2(N_2)\lambda_2^2, ~~~\lambda_1,\lambda_2\geq0,
  \end{split}
\end{equation}
implying that
\begin{equation}
  \begin{split}
    \label{volumeComp}
    \vol_2(N_1,N_2) = \frac{1}{2}(\vol_2(N_1+N_2)-(\vol_2(N_1)+\vol_2(N_2))).
  \end{split}
\end{equation}
Since $N_1$ is an equilateral right triangle of side length one, we have that
$\vol_2(N_1)=\frac{1}{2},$ and because $N_2$ is a line, it follows
that $\vol_2(N_2)=0.$ The polytope $N_1+N_2$ is the Minkowski sum of
the two Newton polytopes $N_1$ and $N_2,$ that is
$N_1+N_2=\conv(\{a+b ~|~ a\in N_1, b\in N_2\}).$ The
Minkowski sum of a line segment and an equilateral right triangle is a
trapezoid, as shown in Figure \ref{fig:KnMink} for $n=4.$
The two bases of the trapezoid have length $\sqrt{2}(n-2)$ and
$\sqrt{2}(n-1),$ and the height of the trapezoid is
$\frac{1}{\sqrt{2}}.$ Hence, the area of $N_1+N_2$ is
\begin{equation}
  \label{MinkVol}
\vol_2(N_1+N_2) = \frac{\sqrt{2}(2n-3)}{2}\cdot\frac{1}{\sqrt{2}}=\frac{2n-3}{2},
\end{equation}
and from (\ref{volumeComp}) we have that $\vol_2(N_1,N_2) =
\frac{n-2}{2}.$ Thus, the coefficient of $\lambda_1\lambda_2$ in
(\ref{MinkowskiExp}) is $n-2,$ which is precisely $\MV(N_1,N_2).$ 
\end{proof}

\begin{prop}
  \label{prop:cellDeathSSD}
For the chemical reaction network $CD_n$ there are $n$ steady
states, including two boundary steady-states. 
\end{prop}


\begin{proof}
Based on the discussion following (\ref{eq:cellDeathGeneral}), we wish to solve a square
polynomial system in two variables with one linear equation and one
equation of degree $n.$ This is easily done with elimination.
Using the linear conservation equation, we can
express one of the indeterminates, say $x_Z,$ in terms of $x_Y,$ that
is $x_Z = c_Y-c_Z-x_Y.$ Observe that we can factor out $x_Yx_Z$ in
$\dot{x}_Y,$ and substitute the expression for $x_Z.$ 
This results in two boundary solutions of the
form $(x_Z,y_Z) = (0, c_Y-c_Z), (c_Y-c_Z,0)$, and $n-2$ complex
solutions in $(\CC^*)^2.$
Hence, there are $n$ steady-states, including the boundary steady-states. 
\end{proof}

\begin{rmk}
  \label{rmk:cellDeath}
Based on Proposition \ref{prop:cellDeathSSD}, there are more steady-states than the mixed volume
predicts. This is not contradictory, since the mixed volume gives a
bound on the solutions in the torus $(\CC^*)^n,$ while the
steady-state degree counts all solutions of the polynomial
system. When there are boundary steady-states, i.e., solutions with
some zero entries, the steady-state degree
may be larger than the mixed volume.
\end{rmk}

\subsection{Edelstein model}
\label{subsec:Edelstein}
The Edelstein model was proposed by B. Edelstein in 1970 \cite{E}. It is known to exhibit multiple
real, positive steady states \cite{M-P-V} and thus is an example of a multistationary network. We study the behavior of the steady-state
degree of the network after gluing $n$ copies of the Edelstein
model over shared complexes (see \cite{G-H-M-S} for more details on
gluing); we denote the new network $E_n$.
This model is of particular interest, because although the B\'ezout
bound is exponential in the number of species, the mixed volume
bound is constant and is achieved for all $n.$ To construct $E_n$, we start with the Edelstein model $E_1$ itself: $\{A
~\substack{\leftarrow\\[-1em]\rightarrow} ~ 2A, A+B
~\substack{\leftarrow\\[-1em]\rightarrow} ~ B_1
~\substack{\leftarrow\\[-1em]\rightarrow} ~ B\}$. Then beginning at $i=2$ and continuing until $i=n$, each step is
defined by adding one new species $B_i$ and four reactions gluing over
the complexes $A+B$ and $B.$ For instance, for $n=2,$ the network
$E_2$ would have the form: $\{A
~\substack{\leftarrow\\[-1em]\rightarrow} ~ 2A, A+B
~\substack{\leftarrow\\[-1em]\rightarrow} ~ B_1
~\substack{\leftarrow\\[-1em]\rightarrow} ~ B, A+B
~\substack{\leftarrow\\[-1em]\rightarrow} ~ B_2
~\substack{\leftarrow\\[-1em]\rightarrow} ~ B\}.$
In general, the $n$th reaction network in this
family has $n+2$ species, $n+4$ complexes, and $4n+2$
reactions. The corresponding polynomial system consists of one
conservation equation and $n+2$ differential
equations:
\begin{equation}
  \begin{split}
    \label{eq:EdelGeneral}
f_1 & = x_B -c_B+\sum_{i=1}^n(x_{B_i}-c_{B_i})\\
f_2 & =
  -k_{10}x_A^2-(k_{23}+k_{25}+\cdots+k_{2,n+3})x_Ax_B+k_{01}x_A+k_{32}x_{B_1}\\
&~~~~ +k_{52}x_{B_2}+\cdots+k_{n+3,2}x_{B_n}\\
f_3 & =
            -(k_{23}+k_{25}+\cdots+k_{2,n+3})x_Ax_B-(k_{43}+k_{45}
+\cdots +k_{4,n+3})x_B\\
&~~~~+(k_{32}+k_{34})x_{B_1}+\cdots+(k_{n+3,2}+k_{n+3,4})x_{B_n}\\
f_4 & = k_{23}x_Ax_B+k_{43}x_B-(k_{32}+k_{34})x_{B_1}\\
&~~\vdots\\
f_{n+3} & =
k_{2,n+3}x_Ax_B+k_{4,n+3}x_B-(k_{n+3,2}+k_{n+3,4})x_{B_n}.
\end{split}
\end{equation} 
Observe that only $n+1$ of the differential equations are needed to define the steady-state system as
there is a linear dependence between
$f_3,\ldots,f_{n+3},$ namely $f_3 = -\sum_{i=4}^{n+3}f_i.$
Despite the exponential B\'ezout
bound shown in Proposition \ref{prop:EdelBezout},
it turns out that the mixed volume of the polynomial system (\ref{eq:EdelGeneral})
is constant and it is achieved as the steady-state degree. 

\begin{prop}
  \label{prop:EdelBezout}
The chemical reaction network $E_n$ has a B\'ezout bound of $2^{n+1}.$
\end{prop}

\begin{proof}
There are $n+3$ equations
in the system, where one equation is linear and the rest $n+2$ are
quadratic. Since $f_3 = -\sum_{i=4}^{n+1}f_i,$ we
can drop the polynomial $f_3$ and we are left with $n+1$ quadratic
equations. This gives us a B\'ezout bound of $2^{n+1}.$
\end{proof}

\begin{thm}
  \label{thm:EdelMV}
The mixed volume of the polynomial system corresponding to $E_n$
is 3. 
\end{thm}

\begin{exm}
  \label{example:edelstein1}
Before we give a proof to Theorem \ref{thm:EdelMV} we give
details for $n=1.$ The polynomial system for $E_1$ is
\begin{equation}
  \begin{split}
    \label{eq:E1}
    f_1 & = x_B+x_{B_1}-c_B-c_{B_1}\\
    f_2 = \dot{x}_A & = -k_{1,0}x_A^2+k_{0,1}x_A -
    k_{2,3}x_Ax_B+k_{3,2}x_{B_1}\\
    f_3 = \dot{x}_B & =
    -k_{2,3}x_Ax_B+k_{3,2}x_{B_1}+k_{3,4}x_{B_1}-k_{4,3}x_B\\
    f_4 = \dot{x}_{B_1} & = k_{2,3}x_Ax_B - k_{3,2}x_{B_1}-k_{3,4}x_{B_1}+k_{4,3}x_B.
  \end{split}
\end{equation}
Let $S_i$ be the support of $f_i, i=1,\ldots,4,$
and $Q_i=\conv(S_i),$ where $f_4 = -f_3,$ so we consider only $f_3.$
For ease of notation we write $101$ for $(1,0,1).$
Then, the supports of the three polynomials are
\begin{equation}
  \begin{split}
    \label{supp:E1}
  S_1 & = \{000,010,001\}\\
  S_2 & = \{200,110,100,001 \}\\
  S_3 & = \{110,010,001 \}.
  \end{split}
\end{equation}
Let $S=S_1\cup S_2\cup S_3,$ and $Q = \conv(S),$ see
Figure \ref{fig:triangEdel3D}. We
will show that the collection of sets in (\ref{supp:E1}) satisfies the
hypothesis of \cite[Theorem 2]{C}.
Let $F$ be a facet of $Q,$ which is a pyramid with a trapezoidal
base. If $F$ is one of the lateral facets,
then $F$ contains $001$ which is a member of each set $S_i, i=1,2,3.$
If $F$ is the base of the pyramid, then $F\cap S_i\neq\emptyset,
i=1,2,3,$ since $F$ contains at least two elements from each set
$S_i.$ If $F$ is an edge containing $001,$ then $F\cap S_i\neq\emptyset,
i=1,2,3.$ The edges containing $001$ are the lateral edges. We now
consider the four edges of the base of $Q.$ In the case when $F =
\conv(110,010),$ we have that $F\cap S_i\neq\emptyset$ for all $i.$
Otherwise, when $F$ is one of the other three edges, condition (B)
of \cite[Theorem 2]{C} is satisfied, since for at least one
$i=1,2,3, F\cap S_i$ is a singleton. Hence, each face of $Q$
satisfies either condition (A) or (B) of \cite[Theorem 2]{C} and
therefore the mixed volume of the system in (\ref{eq:E1}) is the same
as the normalized volume of the convex hull of the union of the Newton
polytopes of the corresponding
system. That is
\begin{equation}
  \label{mv:E1}
\MV(Q_1,Q_2,Q_3) = 3!\vol_3(Q).
\end{equation}
The Euclidean volume of $Q$ is the number of simplices contained in a
unimodular regular triangulation of $Q,$ times
the normalized volume of a
unimodular three-dimensional
simplex, which is $1/3!.$
To see the triangulation, first we note that $Q$ is a pyramid with a
trapezoidal base.
This base has a unimodular triangulation
containing three simplices, see Figure \ref{fig:triangEdelBase}.
To construct $Q$ we simply add the vertex $001$ and
cone over the existing simplices, see Figure
\ref{fig:triangEdel3D}. Hence, there are 3 simplices,
each with volume $1/3!.$ By (\ref{mv:E1}) we have that
\begin{equation}
\MV(Q_1,Q_2,Q_3) = 3!\cdot 3\cdot\frac{1}{3!}=3. \qedhere
\end{equation}
\begin{figure}
  \centering
  \includegraphics[scale=0.3]{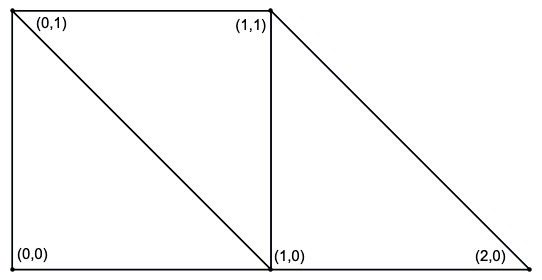}
  \caption{Unimodular triangulation of the trapezoidal base of $Q$ in
    Example \ref{example:edelstein1}.}
    \label{fig:triangEdelBase}
\end{figure}
\begin{figure}
  \centering
  \includegraphics[scale=0.3]{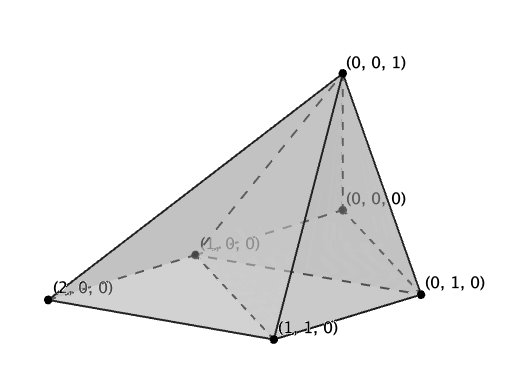}
  \caption{The polytope $Q=\conv(S)$ from Example
    \ref{example:edelstein1} and its unimodular triangulation.}
    \label{fig:triangEdel3D}
\end{figure}
\end{exm}

To prove Theorem \ref{thm:EdelMV}, we use results from \cite{C}
to compute the
mixed volume of the polynomial system in
(\ref{eq:EdelGeneral}). Consider (\ref{eq:EdelGeneral}) and let
$S_i=\newt(f_i).$ We relabel the polynomials $f_i$ by omitting $f_3$
and letting $f_j=f_{i-1}$ for $i\geq4.$ Let $Q_i=\conv(S_i),
i\in\{1,\ldots,n+2\},$ and for $n\geq 2$ let $\wt{S}=S_3\cup\cdots\cup
S_{n+2}, \wt{Q} = \conv(\wt{S}),$ and $\mcQ = \conv(S_{\cup}),$ where
$S_{\cup} = \bigcup_{i=1}^{n+2}S_i.$
\begin{lem}
  \label{lemma:EdelMV}
Let $n\geq 2$ and consider the chemical reaction network $E_n$ and the
corresponding polynomial system (\ref{eq:EdelGeneral}).  
Then
\begin{equation}
  \label{eq:EdelLemma}
  \MV(Q_1,\ldots,Q_{n+2})
  = (n+2)!\vol_{n+2}(\mcQ).
\end{equation}
\end{lem}


\begin{proof}
The mixed volume computation in this case can be reduced to a {\em semi-mixed
  volume} computation, where some of the
polytopes are identical. First, we want to show that
\begin{equation*}
\MV(Q_1,\ldots,Q_{n+2}) =
\MV(Q_1, Q_2, \underbrace{\wt{Q},\ldots,\wt{Q}}_{n}).
\end{equation*}
Let the
indeterminates of (\ref{eq:EdelGeneral}) be ordered lexicographically: $x_A,
x_B, x_{B_1}, \ldots, x_{B_n},$ and let $e_i$ be the corresponding
exponent vetor for each monomial.
We write $e_0$ for the zero vector
in $\RR^{n+2}$ and $e_{ij}$ for $e_i+e_j,$ where
$i,j\in\{A,B,B_1,\ldots,B_n\}.$
After the aforementioned relabeling 
the supports of the $f_i$s, $i\in\{1,\ldots,n+2\},$ in (\ref{eq:EdelGeneral}) are
\begin{align}
    S_1 & = \{ e_0,e_B,e_{B_1},\ldots,e_{B_n}\} \nonumber\\
    S_2 & = \{2e_A, e_{AB},e_A, e_{B_1}\ldots,e_{B_n}\} \nonumber\\
    S_3 & = \{e_{AB},e_B,e_{B_1}\} \label{newSupportEdel}\\
    & \vdots\nonumber\\
    S_{n+2} & = \{e_{AB},e_B,e_{B_n}\}.\nonumber
\end{align}
Observe that $S_3,\ldots, S_{n+2}$ differ by one element only, hence,
they meet the criterion in \cite[Corollary 1]{C} implying that 
\begin{equation}
  \label{eq:proofE1}
  \MV(Q_1,\ldots,Q_{n+2}) = \MV(Q_1,Q_2,\underbrace{\wt{Q},\ldots,\wt{Q}}_{n}).
\end{equation}
Now, the mixed volume of the system with support (\ref{newSupportEdel}) is the
same as the mixed volume of the system below.
\begin{equation}
  \begin{split}
    \wt{S}_1 & = \{ e_0, e_B,e_{B_1},\ldots,e_{B_n}\}\\
    \wt{S}_2 & = \{2e_A,e_A,e_{AB},e_{B_1},\ldots,e_{B_n}\}\\
    \wt{S}_j & = \{ e_{AB}, e_B,e_{B_1},\ldots,e_{B_n}\}, j=3,\ldots,n+2
    \label{equivSupportEdel}
  \end{split}
\end{equation}
We want to show that the collection of $\wt{S}_i, i\in \{1,\ldots,n+2\},$
satisfies the hypothesis of \cite[Theorem 2]{C}. Let $F$ be a
positive dimensional face of $\mcQ.$ If any of the vertices of $F$ are
in $S_{\cap}=\bigcap_{i=1}^{n+2}S_i$ then $F\cap \wt{S}_i\neq\emptyset$ for
all $i\in\{1,\ldots,n+2\}.$ In this case, $F$ satisfies Theorem \ref{ChenThm2}(i). Suppose that none of the vertices of $F$
are in $S_{\cap}.$ Then they must be in the set difference $D=S_{\cup}
\backslash S_{\cap}=\{e_0,e_B,2e_A,e_A,e_{AB} \}.$ Note that
$e_A\in D$ is in the interior of the edge $\{e_0,2e_A\},$ so it
is not a vertex. Suppose that the vertices of $F$ are all of
$D-\{e_A\}.$ Then $F\cap \wt{S}_i\neq\emptyset$ for all $i,$ and we
are in case (i) of Theorem \ref{ChenThm2}. If the vertices of $F$ are a smaller
subset of $D-\{e_A\},$ then $F$ must be an edge. There are four such
edges, and for each one of them, either $F\cap \wt{S}_i\neq\emptyset$
for all $i,$ or for some $j$ we have that $F\cap\wt{S}_j$ is a
singleton. In this case we meet condition (ii) of the theorem. Hence, we have that
\begin{equation}
  \label{EdelLemma}
\MV(Q_1,\ldots,Q_{n+2})=(n+2)!\vol_{n+2}(\mcQ). \qedhere
\end{equation}\end{proof}

\begin{proof}[Proof of Theorem \ref{thm:EdelMV}.]
To compute the volume of $\mcQ$ we construct a unimodular
triangulation. Recall that the Euclidean volume of an $n$-dimensional
unimodular simplex is $\frac{1}{n!}.$ The vertices of $\mcQ$ are
$\{e_0, 2e_A,e_{AB},e_B,e_{B_i}\}$ where $i=1,\ldots,n.$ Note
that all $e_{j}, j\in\{A,B,B_i\}$ are the $\{0,1\}$ unit vectors
in $\RR^{n+2},$ and that vectors $e_{B_i}$ are linearly independent. This
means that we can work with the polytope
$P=\conv(e_0,2e_A,e_{AB},e_B)\subset\RR^2.$
After constructing a unimodular triangulation of $P$ we cone over it
with each of the vertices $e_{B_i}, i\in\{1,\ldots,n\}.$ This process
preserves unimodularity.

As shown in Figure \ref{fig:triangEdelBase}, $P$ is a trapezoid
consisting of three unimodular simplices, each with area
$\frac{1}{2!}=\frac{1}{2}.$ In particular,
we have
\begin{equation}
  \begin{split}
    \sigma_1 & = \conv(00,10,01)\\
    \sigma_2 & = \conv(10,01,11)\\
    \sigma_3 & = \conv(10,20,11).
  \end{split}
\end{equation}
As we cone over the existing triangulation with each $e_{B_i},$ the number
of simplices remains the same; see Figure \ref{fig:triangEdel3D} for
example. Thus, $\mcQ$ has three
$(n+2)$-dimensional simplices, each with volume $\frac{1}{(n+2)!}.$
Hence, by Lemma \ref{lemma:EdelMV} it follows that
\begin{equation}
  \label{EdelMV}
\MV(Q_1,\ldots,Q_{n+2}) = (n+2)!\vol_{n+2}(\mcQ) =
(n+2)!\frac{3}{(n+2)!}=3. \qedhere
\end{equation}
\end{proof}
\begin{thm}
  \label{thm:EdelSSD}
The steady-state degree of the chemical reaction network $E_n$ is 3.
\end{thm}


\begin{proof}
We use elimination to reduce the system to a univariate cubic
polynomial in $x_A;$ the elimination algorithm is easy to see
for$E_1.$ The corresponding system (\ref{eq:E1}) contains four
polynomials in three variables with
$f_4=-f_3$, so we can reduce the system to three
polynomials by forgetting $f_4.$ Using $f_1,$ we
solve for $x_{B_1}$ as a linear expression in $x_B.$ Subtracting $f_3$
from $f_2$ and substituting for $x_{B_1}$ in the difference, we can
solve for $x_B,$ and in turn for $x_{B_1},$ as a quadratic in $x_A.$
Lastly, substituting for all variables in terms of $x_A$ in $f_2$
results in a univariate cubic polynomial in $x_A.$ Hence, there are
exactly three equilibrium solutions to (\ref{eq:E1}). 

The polynomial system
for $n\geq 2$ has the general form of \ref{eq:EdelGeneral}.
Similarly to the first case,
using equations $f_4,\ldots,f_{n+3},$
for each $i=1,\ldots,n$ we can express $x_{B_i}$ as a bilinear
expression in 
 $x_A$ and $x_B.$ These 
expressions can then be substituted in $f_1,$ from where we can solve
for $x_B$ (and respectively all $x_{B_i}$)
as a rational expression in terms of $x_A,$ with a quadratic
numerator and a linear denominator in
$x_A.$ These operations are defined, since we assume that the collection of $k_{ij}$s is generic, and
hence, no linear combination is zero; moreover we assume that $x_A$ is
nonzero. 
Substituting the rational expressions for $x_B$ and $x_{B_i}$ into $f_2$
and clearing the denominators results in a univariate cubic polynomial in $x_A.$
Hence, there are three solutions to the system, i.e., the steady-state
degree is 3. This result along with Lemma \ref{lemma:EdelMV} shows
that the BKK bound is tight for all $n.$
\end{proof}

\subsection{One-site phosphorylation cycle}
\label{subsec:one-site}

The last family of networks we study is based on the one-site phosphorylation cycle, a mechanism that plays a role in the activation
and deactivation of proteins.
In particular, we look at the reaction network $PC_n$ obtained by gluing $n$ one-site distributive phosphorylation cycles over complexes. As an example, when two one-site distributive phosphorylation cycles are glued in this way, we obtain a two-site
phosphorylation cycle \cite{F-W}.

The one-site distributive phosphorylation cycle consists of six species, six complexes, and
six reactions: $\{S_0+E ~\substack{\leftarrow\\[-1em]\rightarrow} ~X_1
\rightarrow S_1+E, S_1+F ~\substack{\leftarrow\\[-1em]\rightarrow} ~Y_1
\rightarrow S_0+F\}.$ The second copy of the one-site phosphorylation cycle will have the
form $\{S_1+E ~\substack{\leftarrow\\[-1em]\rightarrow} ~X_2
\rightarrow S_2+E, S_2+F ~\substack{\leftarrow\\[-1em]\rightarrow} ~Y_2
\rightarrow S_1+F\}$
where all species with index $i$ are replaced by the same type of species
with index $i+1,$ e.g., $S_1$ is replaced by $S_2.$ We glue over the
common complexes $S_1+E$ and $S_1+F.$ For $n$ copies of the cycle, we have $3n+3$ species, $4n+2$ complexes, and $6n$
reactions.  The reaction network $PC_4$ is shown in Figure \ref{f:PCn}.
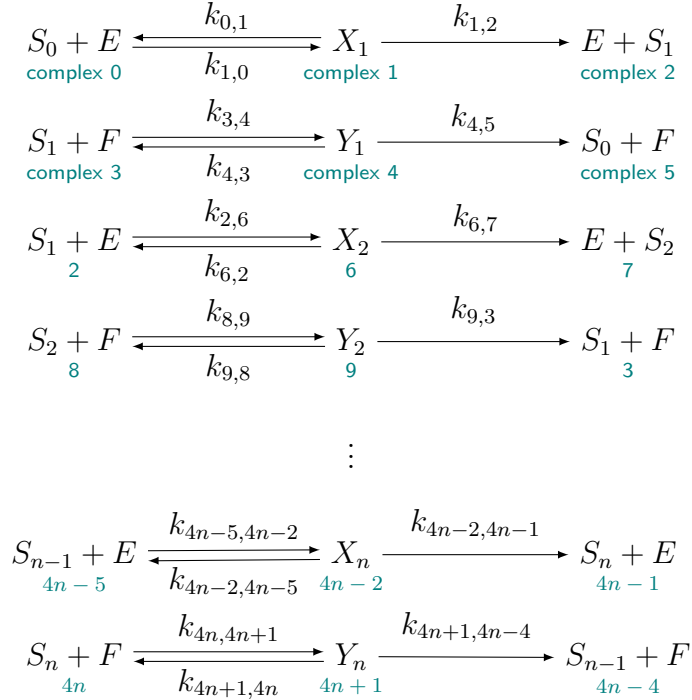
\begin{figure}
    \centering   
        \begin{tikzpicture}
           \matrix (m) [matrix of math nodes, row sep=0.7cm, column sep=2.3cm]
          {S_0+E  & X_1 & E+S_1 \\
            S_1+F & Y_1 & S_0+F \\
            S_1+E & X_2 & E+S_2 \\
            S_2+F & Y_2 & S_1+F \\
            & \vspace{-10em}\vdots & \\
            S_{n-1}+E & X_n & S_n+E \\
            S_n+F & Y_n & S_{n-1}+F \\};          
            \begin{scope}[font=\sffamily]
            \node (complex0) at (m-1-1) [label=below:\scriptsize{{\color{teal}complex 0}}] {};
            
            \node (complex1) at (m-1-2) [label=below:\scriptsize{{\color{teal}complex 1}}] {};
            
            \node (complex2) at (m-1-3) [label=below:\scriptsize{{\color{teal}complex 2}}] {};
            
            \node (complex3) at (m-2-1) [label=below:\scriptsize{{\color{teal}complex 3}}] {};
            
            \node (complex4) at (m-2-2) [label=below:\scriptsize{{\color{teal}complex 4}}] {};
            
            \node (complex5) at (m-2-3) [label=below:\scriptsize{{\color{teal}complex 5}}] {};
            
            \node (complex2) at (m-3-1) [label=below:\scriptsize{{\color{teal}2}}] {};
            
            \node (complex6) at (m-3-2) [label=below:\scriptsize{{\color{teal}6}}] {};
            
            \node (complex7) at (m-3-3) [label=below:\scriptsize{{\color{teal}7}}] {};
            
            \node (complex8) at (m-4-1) [label=below:\scriptsize{{\color{teal}8}}] {};
            
            \node (complex9) at (m-4-2) [label=below:\scriptsize{{\color{teal}9}}] {};
            
            \node (complex3) at (m-4-3) [label=below:\scriptsize{{\color{teal}3}}] {};
            
            \node (complex4n-5) at (m-6-1) [label=below:\scriptsize{{\color{teal}$4n-5$}}] {};
            
            \node (complex4n-2) at (m-6-2) [label=below:\scriptsize{{\color{teal}$4n-2$}}] {};
            
            \node (complex4n-1) at (m-6-3) [label=below:\scriptsize{{\color{teal}$4n-1$}}] {};
            
            \node (complex4n) at (m-7-1) [label=below:\scriptsize{{\color{teal}$4n$}}] {};
            
            \node (complex4n+1) at (m-7-2) [label=below:\scriptsize{{\color{teal}$4n+1$}}] {};
            
            \node (complex4n-4) at (m-7-3) [label=below:\scriptsize{{\color{teal}$4n-4$}}] {};
            \end{scope}
          \path (m-1-1.east) -- (m-1-1.south east) coordinate[pos=-0.2] (a1);
          \path (m-1-2.west) -- (m-1-2.south west) coordinate[pos=-0.2] (b1);
          \draw[latex-] (a1) -- node [below=+0.05] {$k_{1,0}$} (b1);
          \path (m-1-1.east) -- (m-1-1.south east) coordinate[pos=0.2] (a2);
          \path (m-1-2.west) -- (m-1-2.south west) coordinate[pos=0.2] (b2);
          \draw[latex-] (b2) -- node [above=+0.05] {$k_{0,1}$} (a2);
          \draw[latex-] (m-1-3) -- node [above=-0.05] {$k_{1,2}$} (m-1-2);
          
          \path (m-2-1.east) -- (m-2-1.north east) coordinate[pos=-0.2] (c1);
          \path (m-2-2.west) -- (m-2-2.north west) coordinate[pos=-0.2] (d1);
          \draw[latex-] (c1) -- node [below=-0.05] {$k_{4,3}$} (d1);
          \path (m-2-1.east) -- (m-2-1.north east) coordinate[pos=0.2] (c2);
          \path (m-2-2.west) -- (m-2-2.north west) coordinate[pos=0.2] (d2);
          \draw[latex-] (d2) -- node [above=-0.05] {$k_{3,4}$} (c2);
          \draw[latex-] (m-2-3) -- node [above=-0.05] {$k_{4,5}$} (m-2-2);
          
          \path (m-3-1.east) -- (m-3-1.north east) coordinate[pos=-0.2] (c1);
          \path (m-3-2.west) -- (m-3-2.north west) coordinate[pos=-0.2] (d1);
          \draw[latex-] (c1) -- node [below=-0.05] {$k_{6,2}$} (d1);
          \path (m-3-1.east) -- (m-3-1.north east) coordinate[pos=0.2] (c2);
          \path (m-3-2.west) -- (m-3-2.north west) coordinate[pos=0.2] (d2);
          \draw[latex-] (d2) -- node [above=-0.05] {$k_{2,6}$} (c2);
          \draw[latex-] (m-3-3) -- node [above=-0.05] {$k_{6,7}$} (m-3-2);
          
          \path (m-4-1.east) -- (m-4-1.north east) coordinate[pos=-0.2] (c1);
          \path (m-4-2.west) -- (m-4-2.north west) coordinate[pos=-0.2] (d1);
          \draw[latex-] (c1) -- node [below=-0.05] {$k_{9,8}$} (d1);
          \path (m-4-1.east) -- (m-4-1.north east) coordinate[pos=0.2] (c2);
          \path (m-4-2.west) -- (m-4-2.north west) coordinate[pos=0.2] (d2);
          \draw[latex-] (d2) -- node [above=-0.05] {$k_{8,9}$} (c2);
          \draw[latex-] (m-4-3) -- node [above=0.05] {$k_{9,3}$} (m-4-2);
          
          \path (m-6-1.east) -- (m-6-1.north east) coordinate[pos=-0.2] (c1);
          \path (m-6-2.west) -- (m-6-2.north west) coordinate[pos=-0.2] (d1);
          \draw[latex-] (c1) -- node [below=-0.05] {$k_{4n-2,4n-5}$} (d1);
          \path (m-6-1.east) -- (m-6-1.north east) coordinate[pos=0.2] (c2);
          \path (m-6-2.west) -- (m-6-2.north west) coordinate[pos=0.2] (d2);
          \draw[latex-] (d2) -- node [above=-0.05] {$k_{4n-5,4n-2}$} (c2);
          \draw[latex-] (m-6-3) -- node [above=0.05] {$k_{4n-2,4n-1}$} (m-6-2);
          
          \path (m-7-1.east) -- (m-7-1.north east) coordinate[pos=-0.2] (c1);
          \path (m-7-2.west) -- (m-7-2.north west) coordinate[pos=-0.2] (d1);
          \draw[latex-] (c1) -- node [below=-0.05] {$k_{4n+1,4n}$} (d1);
          \path (m-7-1.east) -- (m-7-1.north east) coordinate[pos=0.2] (c2);
          \path (m-7-2.west) -- (m-7-2.north west) coordinate[pos=0.2] (d2);
          \draw[latex-] (d2) -- node [above=-0.05] {$k_{4n,4n+1}$} (c2);
          \draw[latex-] (m-7-3) -- node [above=0.05] {$k_{4n+1,4n-4}$} (m-7-2);
        \end{tikzpicture}
    \caption{A chemical reaction network of type $PC_n$ with labels
      for complexes and notation convention for reaction constants. }
    \label{f:PCn} 
  \end{figure}
The corresponding polynomial system consists of three conservation equations and $3n + 1$ distinct differential equations up to sign.  The three conservation equations are
\begin{align}
    f_1 & = x_E-c_E+\sum_{i=1}^n(x_{X_i}-c_{X_i})  \nonumber\\
    f_2 & = x_F-c_F+\sum_{i=1}^n(x_{Y_i}-c_{Y_i})   \label{eq:one-siteConservation}\\ 
    f_3 & = \sum_{i=0}^n(x_{S_i}-c_{S_i})-(x_E-c_E)-(x_F-c_F), \nonumber
\end{align}
and the $3n+1$ distinct differential equations for $n\geq 2$ are
\begin{equation}
  \begin{split}
    \label{eq:one-siteODEs}
    f_4 & = \dot{x}_{S_0} = -k_{01}x_{S_0}x_E + k_{10}x_{X_1} +
    k_{45}x_{Y_1} \\
    f_5 & = \dot{x}_{S_1} = -k_{26}x_{S_1}x_E - k_{34}x_{S_1}x_F +
    k_{12}x_{X_1} + k_{43}x_{Y_1} + k_{62}x_{X_2} + k_{93}x_{Y_2} \\
    f_{j+4} & = \dot{x}_{S_j} = -k_{4j,4j+1}x_{S_j}x_F +
    k_{4j-2,4j-1}x_{X_j} + k_{4j+1,4j}x_{Y_j} -
    k_{4j-1,4j+2}x_{S_j}x_E \\
    & ~~~~~~~~~~~~ + k_{4j+2,4j-1}x_{X_{j+1}} +
    k_{4j+5,4j}x_{Y_{j+1}}, ~ j = 2,\ldots, n-1 \\
    f_{n+4} & = \dot{x}_{S_n} = -k_{4n,4n+1}x_{S_n}x_F +
    k_{4n-2,4n-1}x_{X_n} + k_{4n+1,4n}x_{Y_n} \\
    f_{n+5} & = \dot{x}_{X_1} = k_{01}x_{S_0}x_E -
    (k_{10}+k_{12})x_{X_1} \\
    f_{n+6} & = \dot{x}_{X_2} = k_{26}x_{S_1}x_E -
    (k_{62}+k_{67})x_{X_2} \\
    f_{n+j+4} & = \dot{x}_{X_j} = k_{4j-5,4j-2}x_{S_{j-1}}x_E -
    (k_{4j-2,4j-5}+k_{4j-2,4j-1})x_{X_j}, ~j = 3,\ldots,
    n \\
    f_{2n+5} & = \dot{x}_{Y_1} = k_{34}x_{S_1}x_F -
    (k_{43}+k_{45})x_{Y_1} \\
    f_{2n+6} & = \dot{x}_{Y_2} = k_{89}x_{S_2}x_F -
    (k_{93}+k_{98})x_{Y_2} \\
    f_{2n+j+4} & = \dot{x}_{Y_j} = k_{4j,4j+1}x_{S_j}x_F -
    (k_{4j+1,4(j-1)}+k_{4j+1,4j})x_{Y_j},  ~j = 3,\ldots, n.
  \end{split}
\end{equation}
The full list of steady-state equations includes
$\dot{x}_E = -\sum_i\dot{x}_{X_i}$ and $\dot{x}_F =
-\sum_i\dot{x}_{Y_i},$ which we disregard, since they are linear
combinations of other polynomials from the system. Let
$\wt{P}_n$ be the polynomial system for the reaction network $PC_n$
consisting of the $3n+4$
equations from (\ref{eq:one-siteConservation}) and
(\ref{eq:one-siteODEs}) set equal to zero. 

\begin{prop} \label{prop:BezoutPC} The B\'ezout bound for the reaction network $PC_n$ is $2^{3n+1}.$
\end{prop}
\begin{proof}
  Note that each of the $3n+1$ polynomial ODEs in (\ref{eq:one-siteODEs}) is
quadratic, and each of the three conservation equations in (\ref{eq:one-siteConservation}) is
linear. Hence, the B\'ezout bound for the system $\wt{P}_n$ is $2^{3n+1}.$
\end{proof}

Since $\wt{P}_n$ is overdetermined, to compute the mixed volume and compare it with the B\'ezout bound, we consider the
randomized system $P_n = M\cdot \wt{P}_n,$ where
$M\in\CC^{(3n+3)\times(3n+4)}$ is a generic matrix. Note that every solution of $\wt{P}_n$ is a solution of $P_n$, so the mixed volume of $P_n$ still provides an upper bound on the number solutions of $\wt{P_n}$ in $(\mathbb C^*)^n$. The system $P_n$
is a square system  with $3n +3$ equations where each polynomial is a linear combination of
the polynomials $f_i$, $i = 1,\ldots, 3n+4.$

For the remainder of this section we work with the system $P_n$
where each polynomial has support $S_n = \bigcup_{i=1}^{3n+4}S_i$
for $S_i=\supp(f_i), \ i=1,\ldots, 3n+4.$ Let $Q_n =
\conv(S_n)$ be the Newton polytope of each polynomial of $P_n.$ This
leads to the main theorem of this section.
\begin{thm}
  \label{thm:one-siteMV}
  Let $P_n$ be the randomized polynomial system for the reaction network $PC_n$. Then,
  \begin{equation}
    \label{eq:one-siteMV}
    \MV(\underbrace{Q_n,\ldots,Q_n}_{3n+3}) =
    (3n+3)!\vol_{3n+3}(Q_n) = \frac{(n+1)(n+4)}{2}-1.
  \end{equation}
\end{thm}
The first equality of (\ref{eq:one-siteMV}) follows from the
definition of mixed volume in the special case when all polytopes are
identical. To prove the second equality we construct a triangulation $T_n$ of the
polytope $Q_n.$ 
Provided $T_n$ is {\em unimodular}, i.e., all simplices are unimodular, the normalized Euclidean volume of
$Q_n$ is the number of simplices in $T_n.$ First we
give a description of the vertices of $Q_n,$ followed by a
hyperplane representation of $Q_n,$ which aids in the construction of
the triangulation $T_n$ with the desired number of simplices. 
We illustrate Theorem \ref{thm:one-siteMV} with an example
for $n=1.$
\begin{exm}
  \label{example1}
  The reaction network for $n=1$ is $\{S_0+E ~\substack{\leftarrow\\[-1em]\rightarrow} ~X_1
\rightarrow S_1+E, S_1+F ~\substack{\leftarrow\\[-1em]\rightarrow} ~Y_1
\rightarrow S_0+F\},$ and the corresponding polynomial system $\wt{P}_1$ is
\begin{equation}
  \begin{split}
    f_1 & = x_E+x_{X_1}-c_E-c_{X_1} \\
    f_2 & = x_F+x_{Y_1}-c_F-c_{Y_1} \\
    f_3 & = x_{S_0}+x_{S_1}-x_E-x_F-c_{S_0}-c_{S_1}+c_{E}+c_{F} \\
    f_4 & = -k_{01}x_{S_0}x_E+k_{10}x_{X_1}+k_{45}x_{Y_1} \\
    f_5 & = -k_{34}x_{S_1}x_F+k_{12}x_{X_1}+k_{43}x_{Y_1} \\
    f_6 & = k_{01}x_{S_0}x_E-(k_{10}+k_{12})x_{X_1} \\
    f_7 & = k_{34}x_{S_1}x_F-(k_{43}+k_{45})x_{Y_1}.
  \end{split}
\end{equation}
We take generic parameters $k_{ij}$ and consider the randomized system
$P_1$ with six equations in six variables with the following order:
$x_{S_0},x_E,x_{X_1},x_{S_1},x_F,x_{Y_1}.$ Each polynomial in $P_1$
has the same support, namely
\[
  S_1 = \left\{
  \begin{pmatrix}
    0\\0\\0\\0\\0\\0 
  \end{pmatrix},
   \begin{pmatrix}
    1\\0\\0\\0\\0\\0 
  \end{pmatrix}, 
  \cdots,
   \begin{pmatrix}
    0\\0\\0\\0\\0\\1 
  \end{pmatrix},
   \begin{pmatrix}
    1\\1\\0\\0\\0\\0 
  \end{pmatrix},
  \begin{pmatrix}
    0\\0\\0\\1\\1\\0 
  \end{pmatrix}
  \right\} =\{e_0,e_1,\ldots,e_6,e_{12},e_{45}\}.
\]
For $Q_1=\conv(S_1)\subseteq\RR^6,$ the mixed volume for the system $P_1$
is
\begin{equation*}
   \MV(\underbrace{Q_1,\ldots,Q_1}_{6}) = 6!\vol_6(Q_1).
 \end{equation*}
Observe that $e_3$ and
$e_6$ are linearly independent from the rest of the vertices as vectors. In order
to simplify computations, we will project away $e_3$ and $e_6$ and relabel the vertices. We will
study the new polytope
$K_1=\conv(\vv_1)$ in $\RR^4$ where
\begin{equation*}
  \vv_1 = 
   \left\{
  \begin{pmatrix}
    0\\0\\0\\0
  \end{pmatrix},
   \begin{pmatrix}
    1\\0\\0\\0 
  \end{pmatrix}, 
  \cdots,
   \begin{pmatrix}
    0\\0\\0\\1 
  \end{pmatrix},
   \begin{pmatrix}
    1\\1\\0\\0 
  \end{pmatrix},
  \begin{pmatrix}
    0\\0\\1\\1 
  \end{pmatrix}
  \right\} =  \{v_0,v_1,\ldots,v_4,v_{12},v_{34}\}.
\end{equation*}
Then we will cone over the triangulation of $K_1$ with $e_3$ and then $e_6$ to recover $Q_1.$ 
To compute the
volume of $K_1$ we construct a placing triangulation $\mcT_1,$ which
is unimodular; see the proof of Lemma \ref{l:one-siteT} and \cite{dL, L-S} for
more details.

We begin the triangulation by placing the first five vertices
$v_0,\ldots,v_4,$ which form a standard simplex in $\RR^4.$ Let
$\sigma_1=\conv(v_0,\ldots,v_4).$ Next we place the vertex $v_{12}.$
Note that $v_{12}\not \in \sigma_1,$ but it is in the affine hull of
$\sigma_1$. We consider the facets of $\sigma_1$
visible from $v_{12},$ where the only such facet is $F_1 =
\conv(v_1,\ldots,v_4)$ since all other facets lie on the coordinate
hyperplanes.  
We cone over $F_1$ with $v_{12}$ and obtain the simplex
$\sigma_2=\conv(v_1,\ldots,v_4,v_{12})$. Lastly, we place $v_{34}$ and
observe that $v_{34}$ is not in the convex hull of $\{v_0, v_1,\ldots,v_4,v_{12} \}$ but it is in their
affine hull. None of the facets of $\sigma_1$ are
visible from $v_{34},$ but two of the facets of $\sigma_2$ are
visible: $F_{21} = \conv(v_1,v_3,v_4,v_{12})$ and $F_{22} =
\conv(v_2,v_3,v_4,v_{12}).$
We cone over each one with $v_{34}$ constructing two more simplices:
$\sigma_3 = \conv(F_{21}\cup\{v_{34}\})$
and $\sigma_4 = \conv(F_{22}\cup\{v_{34}\}).$ The collection
$\mcT_1=\bigcup_{i=1}^4\sigma_i$ is a triangulation of $K_1$ by
construction. Moreover, by a similar proof as the one for Lemma \ref{l:one-siteT}, $\mcT_1$ is a
unimodular triangulation.

To construct a triangulation of $Q_1$, we embed $K_1$ in $\mathbb R^6$ and then we cone over each $\sigma_i$ with
$e_3$ and then $e_6$. This gives $T_1 =
\bigcup_{i=1}^4s_i,$ where
\begin{align*}
  s_1 & = \conv(e_0,\ldots,e_6), \\
  s_2 & = \conv(e_1,\ldots,e_6,e_{12}), \\
  s_3 & = \conv(e_1,e_3,e_4,e_5,e_6,e_{12},e_{45}),\\
  s_4 & = \conv(e_2,e_3,e_4,e_5,e_6,e_{12},e_{45}).
\end{align*}
The triangulation $T_1$ remains unimodular, hence the normalized
Euclidean volume of each
simplex is $1/6!,$ and 
\begin{equation*}
  \MV(\underbrace{Q_1,\ldots,Q_1}_6) = 6!\vol_6(Q_1) =
  6!\cdot\frac{4}{6!} = 4 = \frac{(n+1)(n+4)}{2}-1. \qedhere
\end{equation*}
\end{exm}

Now lets consider the general case where the dimension of the ambient space of $Q_n$
is $3n+3$.
Let $e_i\in\RR^{3n+3}$
represent the $i$th standard unit vector, $e_0$ be the zero
vector, and $e_{ij}=e_i+e_j.$ For $1 \leq i \leq d$, the vector $e_i$ is the
exponent vector of $i$th indeterminate in the following ordered list
$(x_{S_0},x_E,x_{X_1},x_{S_1},x_F,x_{Y_1},x_{X_j},x_{S_j},x_{Y_j})_{j=2}^n$
of size $3n+3.$ For $n=1$ and
$n=2$ 
the vertices of $Q_1$ and $Q_2$ are given by the vector
configurations $e_0,e_1,\ldots,e_6,e_{12},e_{45}$ and
$e_0,e_1,\ldots,e_9,e_{12},e_{24},e_{45},e_{58},$ respectively. Going from the $(j-1)$-site phosphorylation network to the $j$-site phosphorylation network ($j \geq 2$), we gain three new steady-state equations and five new monomials:
$x_{X_j},x_{S_j},x_{Y_j},x_{S_{j-1}}x_E,x_{S_j}x_F.$ Hence, for $n\geq3$
the vertices of $Q_n$ are given by the $5n+4$ vectors of dimension $3n+3$
in the configuration
\begin{equation}
  \begin{split}
    \label{eq:vertices}
  V_n = \{e_0,e_1,\ldots,e_{3n+3},e_{12},&e_{24},e_{28},\ldots,
  e_{2,3n-7},e_{2,3n-4},\\
  & e_{45},e_{58},\ldots,
  e_{5,3n-1},e_{5,3n+2}\}.
  \end{split}
\end{equation}
\begin{prop}
  \label{l:one-siteH1}
  Let $Q_n$ be the Newton polytope of each polynomial in the
  system $P_n$ for $n\geq2.$ The $\mathcal{H}$-representation of
  $Q_n$ is given by
   \begin{equation}
    \begin{split}
      \label{eq:H1rep}
    1-x_1-x_3-x_4-\sum_{i=6}^{3n+3}x_i & \geq 0 \\
    1-x_1-x_3-x_5-\sum_{i=2}^{n}(x_{3i}+x_{3i+1})-x_{3n+3} & \geq 0 \\
    1-x_2-x_3-x_5-\sum_{i=2}^n(x_{3i}+x_{3i+1})-x_{3n+3} & \geq 0 \\
    1-x_2-x_3-\sum_{i=2}^n(x_{3i}+x_{3i+1})-x_{3n+2}-x_{3n+3} & \geq
    0 \\
    x_i & \geq 0, ~ i = 1,\ldots,3n+3.
    \end{split}
    \end{equation}
\end{prop}
\begin{proof}
  Let $Q_n^{\mcH}$ be the polytope defined by (\ref{eq:H1rep}). We aim to
  show that $Q_n$ and $Q_n^{\mcH}$ coincide. Note that each coordinate 
  $x_i, i\in\{1,\ldots,3n+3\}$ is bounded in $Q_n^{\mcH}$; in
  particular, $0\leq x_i\leq 1.$ Otherwise,  if $x_i> 1$  (or $x_i< 0$) at
  least one of the multivariate (resp. univariate) inequalities will
  be violated. It remains to show that the vertex
  sets of $Q_n$ and $Q_n^{\mcH}$ coincide. Observe that none of the
  inequalities in (\ref{eq:H1rep}) can be obtained by taking positive
  linear combinations of the remaining inequalities, implying that
  (\ref{eq:H1rep}) is an {\em irredundant} description of $Q_n^{\mcH},$
  hence each inequality defines a distinct facet \cite{Z}.


  A vertex of the
  polytope $Q_n^{\mcH}$ must be in the intersection of at least $3n+3$
  hyperplanes described in (\ref{eq:H1rep}). Hence, a vertex must satisfy a subsystem of
  (\ref{eq:H1rep}) of size at least $(3n+3)\times(3n+3)$ at equality. We begin by considering subsets of $3n+3$
  inequalities whose corresponding linear systems are consistent. 

  First, consider all $3n+3$ univariate equations and set $x_i=0$ for
  all $i=1,\cdots,3n+3;$ this yields the origin $e_0$ as a vertex of $Q_n^{\mcH}.$
  Next, select $3n+2$ variables $x_i$ set equal to zero and one of
  the four multivariate equations. Note that some of these
  combinations will result in an
  inconsistent system. Those yielding a consistent system
  will have a solution with each coordinate zero except one of the
  $x_i$s, which will be 1; there are $3n+3$ distinct choices for the nonzero $x_i$. These choices yield the
  vertices $e_1,\ldots, e_{3n+3}.$ Thus far we have found $3n+4$ vertices of $Q_n^{\mcH}$
  and each is also a vertex of $Q_n.$

  Continuing in the same manner, we now choose $3n+1$ variables $x_i$ set
  equal to zero and two of the four multivariate equations. Each of
  the nonzero variables must take the value 1. Otherwise we would have
  $0<x_i,x_j<1,$ where $i\neq j,$ implying that they appear together
  in both multivariate equations. In this case, each multivariate
  equation is reduced to $1-x_i-x_j=0.$ However, this system yields a
  positive-dimensional face of $Q_n^{\mcH}$ and hence does not describe a
  vertex. Thus, both nonzero variables must be 1, and they cannot appear in the same multivariate
  equation. Independent of the choice of the two multivariate equations,
  the pair $\{x_i,x_j\}$ will be a subset
  of the variables in the symmetric difference of their supports. In particular, there are $2n$ distinct such choices:
  $\{x_1,x_2\}, \{x_2,x_4\}, \{x_4,x_5\}, \{x_2,x_{3j+2}\},
  \{x_5,x_{3k+2}\},$ for $2\leq j\leq n-1, 2\leq k\leq n.$
  These combinations yield the $2n$ vertices
  $e_{12},e_{24},e_{28},e_{2,11},\ldots,e_{2,3n-1},e_{45},e_{58},\allowbreak\ldots,e_{5,3n+2}.$
  Together with the previously found $3n+4$ vertices, we have a total
  of $5n+4$ vertices of $Q_n^{\mcH}$, which are exactly the vertices of
  $Q_n$ shown in (\ref{eq:vertices}). It remains to show that $Q_n^{\mcH}$ does not have any more
  vertices.

  Suppose that $Q_n^{\mcH}$ has a vertex $q\not\in V_n.$ Then,
  since we considered all vertices of $Q_n^{\mcH}$ with zero, one, or two nonzero entries,
  $q$ must have more than two nonzero entries.  Now suppose that for distinct $i,j,$
  and $k$, the entries
$  q_i,q_j,$ and $q_k$ are
  all nonzero, and the remaining $3n$ entries of $q$ are zero. Note that
  $q_i,q_j,$ and $q_k$ must have value 1, otherwise $q$ cannot satisfy a zero-dimensional
  system constructed from the inequalities in (\ref{eq:H1rep}). Since $q_i=q_j=q_k$ 
 , the variables $x_i$, $x_j$, and $x_k$ cannot appear in the same inequality. But there is no
  possible choice for three such variables, implying it is also not
  possible to have more than three nonzero variables. Therefore, we have found
  all vertices of $Q_n^{\mcH};$ in particular, they coincide with
  the vertex representation of $Q_n,$ hence $Q_n^{\mcH}=Q_n.$
\end{proof}

Now we will compute the normalized Euclidean volume of $Q_n$ by constructing a unimodular triangulation. Let $d_n = n+3$. Similar to Example  \ref{example1} we can reduce $Q_n$ to a
lower-dimensional polytope $K_n\subset\RR^{d_n}$ 
by projecting down $2n$ dimensions
corresponding to the vectors $e_3,e_6,e_{3j+1},$ and $e_{3j+3}, 2\leq
j\leq n.$ These are the exponent vectors of the monomials $x_{X_j}$
and $x_{Y_j}$.  To
avoid ambiguity of notation, we relabel the standard unit vectors and their sums
after the projection (e.g. $v_1$ will be the $1$st standard unit vector in $\RR^{d_n}$ and $v_{12} = v_1 + v_2$), 
so $K_n=\conv(\mcV_n)$ where $|\mcV_n|=3n+4$ and 
\begin{equation}
  \label{eq:verticesProj}
\mathcal{V}_n = \{ v_0, v_1, \ldots, v_{d_n}, v_{12}, v_{23}, v_{25}, \ldots,
v_{2,d_{n-1}}, v_{34}, v_{45}, \ldots, v_{4,{d_n}}\}.
\end{equation}
Following the ideas of Example \ref{example1}, we construct a {\em
  placing} triangulation $\mathcal{T}_n$ of $K_n.$ Then we cone over $\mathcal{T}_n$
with the $2n$ remaining unit vectors from $V_n$ to recover a 
unimodular triangulation of $Q_n$.

We will construct $\mathcal T_n$ by successively placing vertices. After placing each vertex,
we will need information about the convex hull of the vertices already placed.
The following lemma describes these intermediate polytopes and is used in the construction of
$\mcT_n.$ The proofs are omitted as they follow the same process and reasoning as the proof
of Proposition \ref{l:one-siteH1}.
\begin{lem}
  \label{l:HrepKns}
  Let $d_n = n+3$. For each $n$, let $K_{n-1}'$ be the embedding of $K_{n-1}$ in $\mathbb R^{d_n}$. Let $K_n^* = \conv(K_{n-1}'\cup\{v_{d_n}\})$ and
  $\wt{K}_n=\conv(K_n^*\cup\{v_{2,d_{n-1}}\}).$
  Then:
  \begin{enumerate}
    \item The $\mcH$-representation of $K_n^*$ is 
  \begin{equation}
    \begin{split}
      \label{eq:HrepKn*}
      1 - x_1 - x_3 - \sum_{j=2}^nx_{d_j} & \geq 0\\
      1 - x_1 - x_4 - x_{d_n} & \geq 0 \\
      1 - x_2 - x_4 - x_{d_n} & \geq 0\\
      1 - x_2 - x_{d_{n-1}} - x_{d_n} & \geq 0\\
      x_i & \geq 0, i = 1,\ldots, d_n = n+3.
    \end{split}
  \end{equation}
 \item The $\mcH$-representation of $\wt{K}_n$ is 
    \begin{align}
        1-x_1-x_3-\sum_{j=2}^nx_{d_j} & \geq 0 \label{eq:HKn1} \\
        1-x_1-x_4-x_{d_n} & \geq 0 \label{eq:HKn2}\\
        1-x_2-x_4-x_{d_n} & \geq 0 \label{eq:HKn3}\\
        x_i & \geq 0, i=1,\ldots,d_n. \label{eq:HKn4}
    \end{align}
    \end{enumerate}
  \end{lem}
\begin{lem}
  \label{l:one-siteT}
  Let $n\geq 2$ and $d_j=j+3, j\leq n.$ Let $\mathcal{T}_1$ be the triangulation of $K_1$ as described in Example \ref{example1}. Let $\mathcal{T}_n$ be the placing triangulation obtained from $\mcT_{n-1}$ by coning over the $k_{n-1}$ simplices of $\mcT_{n-1}$ with apex $v_{d_n}$ and placing  $v_{2,d_{n-1}}$ and $v_{4,d_n}$, in that order.  The simplices obtained by placing $v_{2,d_{n-1}}$ and $v_{4,d_n}$ are
  \begin{equation} \label{eq:simplices}
    \begin{split}
      \sigma_{k_{n-1}+1} & = \conv( v_2,v_{d_{n-1}},v_{d_n},
      v_{12},v_{23},v_{25},{\ldots,v_{2,d_{n-1}}},v_{4,d_{n-1}} ) \\
      \sigma_{k_{n-1}+2} & = \conv( v_1,v_4,v_{d_n},v_{12},v_{34},v_{45},\ldots,v_{4,d_n} ) \\
      \sigma_{k_{n-1}+3} & = \conv( v_2, v_4, v_{d_n}, v_{12}, v_{23},
      v_{25},\ldots, v_{2,d_{n-1}}, v_{4,d_n} ) \\
      \sigma_{k_{n-1}+4} & = \conv( v_2, v_{d_n}, v_{12}, v_{23},
      v_{34},v_{45},\ldots, v_{4,d_n} ) \\
      \sigma_{k_{n-1}+5} & = \sigma_{k_{n-1}+4}\backslash\{v_{34}\}\cup\{v_{25}\}\\
      & \vdots \\
      \sigma_{k_{n-1}+d_{n-1}} & = \sigma_{k_{n-1}+d_{n-1}-1}\backslash \{v_{4,d_{n-1}-1}\}
      \cup \{v_{2,d_{n-1}}\}. 
    \end{split}
  \end{equation}
Furthermore, $\mcT_n$ has $k_n=4+\sum_{j=2}^{n-1}d_j$ simplices and is unimodular.
\end{lem}

\begin{proof}
  The triangulation $\mathcal T_n$ is obtained
inductively beginning with the explicit construction of
$\mathcal{T}_1$ in Example \ref{example1} containing $k_1 =4$ unimodular
simplices. Suppose the triangulation $\mcT_{n-1}$ has been constructed by
successively placing vertices as described in the statement of the lemma.  Furthermore,
assume $\mcT_{n-1}$ contains $k_{n-1}=4+\sum_{j=2}^{n-2}d_j$ unimodular
simplices as described in \eqref{eq:simplices}.
We embed $K_{n-1}$ and its triangulation $\mcT_{n-1}$ into $\RR^{d_n}$
and place the vertices (i) $v_{d_n}$, (ii) $v_{2,d_{n-1}},$ and (iii) $v_{4,d_n}$ as follows.
\begin{enumerate}
  \item[(i)] Placing $v_{d_n}:$ Placing $v_{d_n}$  increases the dimension of the polytope $K_{n-1}$ by one from $d_{n-1}$ to $d_n$. 
  We cone over all simplices of $\mcT_{n-1}$ with $v_{d_n}$ and obtain the first $k_{n-1}$ simplices of
    $\mcT_n.$ The resulting polytope is $K_n^*$ and its facet defining
    inequalities are given in (\ref{eq:HrepKn*}).
  \item[(ii)] Placing $v_{2,d_{n-1}}:$ Consider the facet defining
      inequalities of $K_n^*$ in (\ref{eq:HrepKn*}). Note that the
      hyperplane $1 - x_2-x_{d_{n-1}}-x_{d_n}=0$ is the only one
      separating $v_{2,d_{n-1}}$ and $K_n^*.$ Facets of $K_n^*$
      contained in this hyperplane will be visible from
      $v_{2,d_{n-1}}.$ There is only one such facet, namely
      \begin{equation*}
        F_{2,d_{n-1},d_n} =
        \conv(v_2,v_{d_{n-1}},v_{d_n},v_{12},v_{23},v_{25},\ldots,v_{2,d_{n-2}},v_{4,d_{n-1}}),
        \end{equation*}
      containing $d_n$ vertices and hence it is a simplex of dimension $d_{n-1}.$
     Coning over $F_{2,d_{n-1},d_n}$ with $v_{2,d_{n-1}}$ yields the
      $d_n$-dimensional simplex $\sigma_{k_{n-1}+1}$. The resulting polytope after placing $v_{2,d_{n-1}}$ is
      $\wt{K}_n$ whose facet defining inequalities are given in (\ref{eq:HKn1}) -- (\ref{eq:HKn4}).
      \item[(iii)] Placing $v_{4,d_n}:$ We aim to show that in this step
        we add $d_{n-1}-1$ new simplices.
        Investigating the facet defining inequalities of $\wt{K}_n$, we
        note that there are two hyperplanes separating $v_{4,d_n}$
        from $\wt{K}_n,$ namely (\ref{eq:HKn2}) and (\ref{eq:HKn3})
        containing the respective facets
        \begin{align*}
        F_{1,4,d_n} &
                      =\conv(v_1,v_4,v_{d_n},v_{12},v_{34},v_{45},\ldots,v_{4,d_{n-1}})\\
        F_{2,4,d_n} &
                      =\conv(v_2,v_4,v_{d_n},v_{12},v_{23},v_{25},\ldots,v_{2,d_{n-1}},v_{34},v_{45},{\color{red}\ldots},
                      v_{4,d_{n-1}}). 
        \end{align*}
        Note that $F_{1,4,d_n}$ is a $d_{n-1}$-dimensional
        simplex, so coning over it with $v_{4,d_n}$ results in the
        $d_n$-dimensional simplex $\sigma_{k_{n-1}+2}.$ 
        
        The facet $F_{2,4,d_n}$ lies in the facet defining hyperplane $1 - x_2 - x_{4} - x_{d_n}=0$; it has
        $2d_{n-1}-2$ vertices and a unimodular triangulation induced by the triangulation of $\wt{K}_n$. In particular, the simplices in the triangulation of $F_{2,4,d_n}$ are 
  $\sigma_{k_{n-2}+3}\setminus \{v_{d_{n-1}}\} \cup \{v_{d_{n}}\},\ldots,\sigma_{k_{n-2}+d_{n-2}} \setminus \{v_{d_{n-1}}\} \cup \{v_{d_{n}}\}$, $\sigma_{k_{n-1}+1} \setminus \{v_{d_{n-1}}\}$.  These $d_{n-1}$ dimensional simplices are obtained by considering the intersection of the simplices $\sigma_{k_{n-2}+1} \cup \{v_{d_{n}}\},\ldots,\sigma_{k_{n-2}+d_{n-2}} \cup \{v_{d_{n}}\},$ $\sigma_{k_{n-1}+1}, \sigma_{k_{n-1}+2}$ of $\wt{K}_n$ with the hyperplane $1 - x_2 - x_{4} - x_{d_n} = 0$;  note that we do not need to consider the remaining simplices of $\wt{K}_n$, since each intersection with $F_{2,4,d_n}$ is necessarily of dimension less than $d_{n-1}$. 

We cone
       over the triangulation of $F_{2,4,d_n}$ with $v_{4,d_n}$ and obtain
       the $d_{n-2}-1 = d_n-3$ simplices
       $\sigma_{k_{n-1}+3},\ldots,\sigma_{k_{n-1}+d_{n-1}}$.  Hence, we have a total of
      \begin{equation*}
        k_n = k_{n-1}+2+(d_n-3) = k_{n-1}+d_{n-1} = 4+\sum_{j=2}^{n-1}d_j
        \end{equation*}
      simplices in $\mcT_n.$ 
    \end{enumerate}

Finally, we show that the placing triangulation $\mcT_n$ is
      unimodular.  The polytope $K_n$ is a $d_n$-dimensional
{\em compressed} polytope \cite{dL},
implying that all of its {\em pulling} triangulations are
unimodular. A placing triangulation is equivalent to a pushing
triangulation. The latter is a regular triangulation with a lifting
vector of heights $\omega: J\rightarrow\RR,$ where $J$ is the set of
labels on $\mathcal{V}_n$ with respect to some order. Reversing the order of the
labels of $\mathcal{V}_n$ and the heights of the weight vector
$\omega$ makes the pushing
triangulation into a pulling triangulation \cite{L-S, dL}. Hence, $\mathcal{T}_n$ as
constructed is a regular unimodular triangulation. 
\end{proof}


\begin{proof}[Proof of Theorem \ref{thm:one-siteMV}.]
The first equality in (\ref{eq:one-siteMV}) follows from the
definition of mixed volume in the special case
when all polytopes are identical. We aim to obtain a unimodular
triangulation of $Q_n.$
  By Lemma \ref{l:one-siteT} $K_n$ has a triangulation $\mcT_n$ with
\begin{equation*}
  4+\sum_{i=1}^{n-1}d_i = 4+\sum_{i=1}^{n-1}i+3 = \frac{(n+4)(n+1)}{2}-1
\end{equation*}
simplices. To achieve a unimodular triangulation of $Q_n,$ we cone over the triangulation
$\mcT_n$ in the $2n$ originally-collapsed dimensions, which preserves
the number of simplices. 
The polytope $Q_n$ has dimension $3n+3,$
hence the normalized Euclidean volume of each full dimensional
unimodular simplex is $\frac{1}{(3n+3)!}.$ The second equality of
(\ref{eq:one-siteMV}) now follows. 
\end{proof} 

The mixed volume for the randomized system of
$PC_n$ is quadratic in $n,$ which
is a tighter bound than the exponential B\'ezout bound. Nonetheless, for 
it is still significantly higher than the steady-state
degree of the ideal that we witness in computation. Indeed, based on numerical computations up to $n=15$, we 
conjecture the following for the steady-state degree of $PC_n$, which is linear in $n$. 

\begin{conj} \label{conj:ssdPC}
The steady-state degree of the chemical reaction network $PC_n$ is
$2n+1$. 
\end{conj}

 \begin{rmk}
 We note the authors of \cite{Wang2008} show that the number of real positive solutions is bounded above by $2n-1$ by using a positive reparameterization.  Along the way they introduce a polynomial with degree $2n+1$. With careful treatment, we expect this polynomial could be used to establish steady-state degree of $PC_n$. 
 \end{rmk}


\medskip 

Our exploration of $Q_n$ reveals that Newton polytopes of steady-state equations are interesting combinatorially on their own. Indeed, we finish our discussion of $Q_n$ by showing that it is a matching polytope of a graph.

Let $G_n$ be the multigraph on $n+3$ vertices with $d=3n+3$ edges, such that
$G_n$ contains one four-cycle, $n-1$ edges incident with one node of
the four-cycle, say $s_1,$ and $2n$ parallel edges connecting $s_1$ diagonally with $s_3.$ See Figure \ref{fig:Gn} for example. Each edge of $G_n$ 
represents a species of $PC_n.$ 

\begin{figure}
  \vspace{-5em}
  \begin{multicols}{2}
  \centering
  \includegraphics[width=4cm]{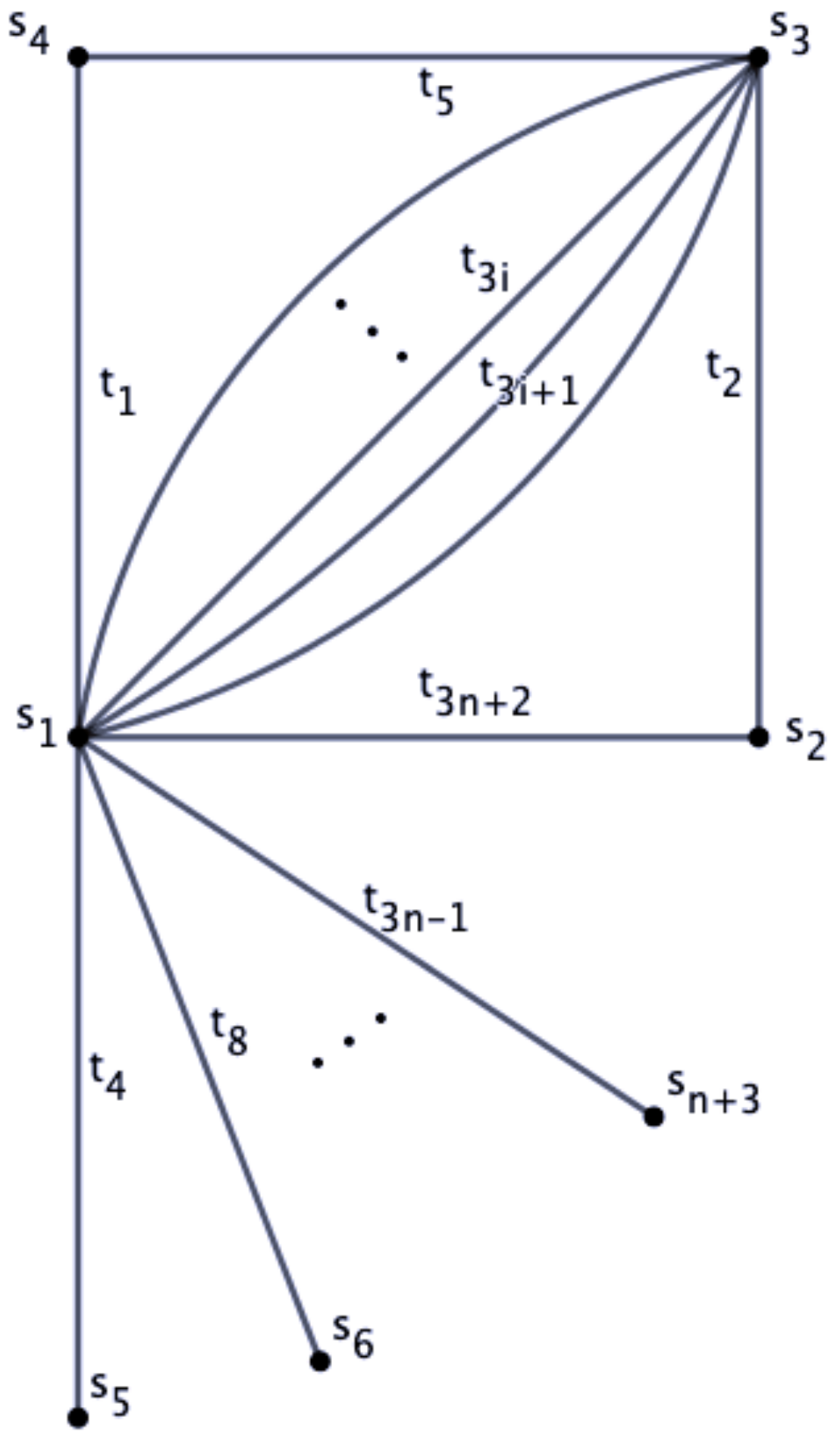}
  \caption{The graph $G_n; Q_n = P_{MA}(G_n).$}
  \label{fig:Gn+}

  \includegraphics[width=3.9cm]{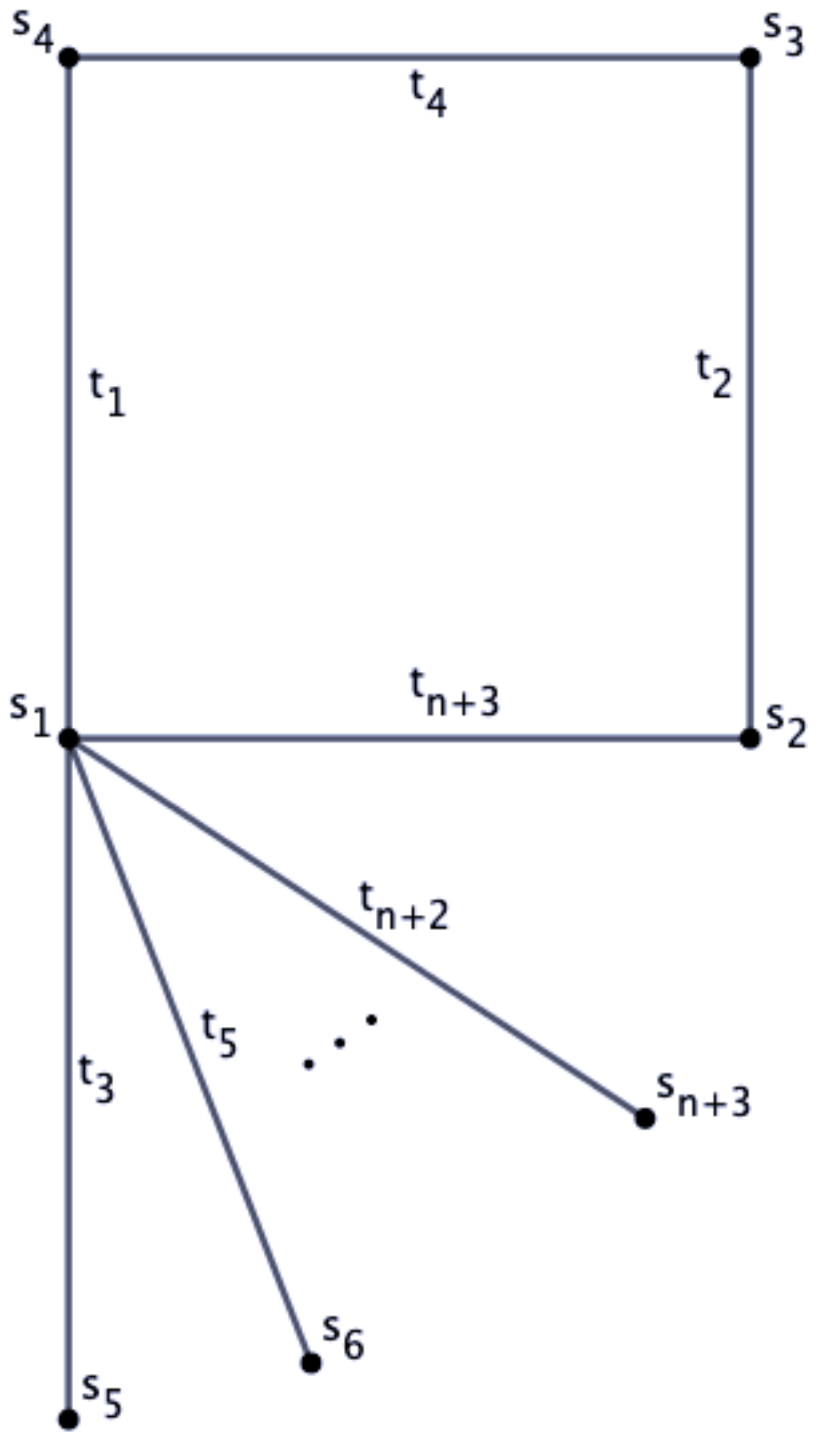}
  \caption{ The graph $\wt{G}_n; K_n = P_{MA}(\wt{G}_n).$}
  \label{fig:Gn}
  \end{multicols}
\end{figure}

The {\em matching polytope} of the graph $G_n$ is the convex hull of
the incidence vectors of all matchings of $G_n,$ i.e.,
\begin{equation*}
  P_{MA}(G_n) = \conv\{\chi^M | M \text{~is~a~matching~} of G_n\}.
\end{equation*}
A {\em matching} of $G_n$ is a subset of edges $M\subseteq E(G_n)$ such
that each vertex is incident with no more than one edge of $M.$ The
{\em incidence} vector $\chi^M\in\{0,1\}^{|E(G_n)|}$ of a matching $M$
is
\begin{equation*}
  \chi^M_t = \left\{ \begin{array}{ll}
                       1, & t\in M, \\
                       0, & \text{otherwise}.
                     \end{array}
                     \right.
                   \end{equation*}
Each matching of the graph $G_n,$ equivalently each vertex of
$P_{MA}(G_n),$ corresponds to the support of a monomial in the
dynamical polynomial
system $P_n$ of \S \ref{subsec:one-site}.          

\begin{prop}\label{matchPolyQn}
  The polytope $Q_n$ is the matching polytope of the graph $G_n$
  described above, i.e., $Q_n = P_{MA}(G_n).$ 
\end{prop}

\begin{proof} 
  The incidence vector of a matching of $G_n$ containing only one edge $t_i$
  coincides with the standard vector $e_i$ with entry 1 in the $i$th position. A matching of $G_n$ can contain at most two
  edges, and each pair is either of the form $M_j = \{t_2,t_j\}, j = 1, 4, 8, 11,
  \ldots, 3n-1$ or of the form $M_{\ell} = \{t_5,t_{\ell}\}, \ell = 4, 8, 11,
  \ldots, 3n-1,3n+2.$ Note that the incidence vectors of the matchings of type
  $M_j$ and $M_{\ell}$ can be represented as $e_{2,j}$ or $e_{5,\ell}$
  for $\ell, j$ as specified above. Hence, the vertices of the
  matching polytope of $G_n$ are the same as the vertices of $Q_n$ as
  given in (\ref{eq:vertices}), implying the two polytopes coincide. 
\end{proof}

Let $\wt{G}_n$ be the simple graph arising from $G_n$ by deleting the $2n$ parallel edges $t_{3i}, t_{3i+1}, 1\leq i\leq n$ and relabeling the remaining edges. Then $\wt{G}_n$ is the graph on $n+3$ vertices and $n+3$ edges. 

\begin{prop}\label{matchPolyKn}
 The polytope $K_n$ is the matching polytope for $\wt{G}_n,$ i.e., $K_n = P_{MA}(\wt{G}_n).$
\end{prop}

\begin{proof}
  Similarly to the proof of Proposition \ref{matchPolyQn}, we will
  show that the vertices of $K_n$ and $P_{MA}(\wt{G}_n)$ are the
  same. Note that the matching for $\wt{G}_n$ will be a subset of the matching of $G_n.$ In particular, there will be $2n$ fewer singleton matchings resulting from the deletion of the $2n$ parallel edges. No two-edge matching will be lost in the construction of $\wt{G}_n$ from $G_n.$ All single edge matchings correspond to the standard vectors $e_i$ for $1\leq i\leq n+3,$ with $e_0$ representing the empty matching. As in the proof of Proposition \ref{matchPolyQn}, we have two-edge matchings of types $M_{j'}$ and $M_{\ell'}$  for $j' = 1, 3, 5, 6, \ldots, n-1$ and $\ell' = 3, 5, 6, \ldots, n$ corresponding to the vertices $v_{2,j'}$ and $v_{4,\ell'}$ from the vertex representation of $K_n$ given in (\ref{eq:verticesProj}). Hence, $K_n = P_{MA}(\wt{G}_n).$
\end{proof}


\section{Acknowledgements}  Elizabeth Gross was supported by NSF
DMS-1620109. Cvetelina Hill was partially supported by NSF DMS-1600569. Additionally, this material is based upon work supported by the National Science Foundation under Grant No. DMS-1439786 while the authors were in residence at the Institute for Computational and Experimental Research in Mathematics in Providence, RI, during the Fall 2018 semester.

\bibliography{refs}
\bibliographystyle{alpha}

\end{document}